\documentclass[12pt]{amsart}
\usepackage{amsmath}
\usepackage{amsthm}
\usepackage{amssymb}
\usepackage{verbatim}
\usepackage{color}

\newtheorem{theorem}{Theorem}    
\newtheorem{proposition}[theorem]{Proposition}

\newtheorem{lemma}[theorem]{Lemma}

\newtheorem{remark}[theorem]{Remark}
\newtheorem{definition}[theorem]{Definition}
\theoremstyle{definition}

\numberwithin{theorem}{section} \numberwithin{theorem}{section}
\numberwithin{equation}{section}

\def\esss{\operatornamewithlimits{ess\,sup}}

\allowdisplaybreaks

\begin{document}
\title[Matrix weighted inequalities for fractional type integrals]
{Matrix weighted inequalities for fractional type integrals associated to operators with new classes of weights}

\author{Yongming Wen$^{\ast}$ and Huoxiong Wu}


%
\thanks{$^*$Corresponding author.}
\thanks{Yongming Wen is supported by National Natural Science Foundation of China (Grant no. 12301119), Fujian Provincial Natural Science Foundation of China(Nos. 2021J05188), President's fund of Minnan Normal University (No. KJ2020020), Institute of Meteorological Big Data-Digital Fujian, Fujian Key Laboratory of Data Science and Statistics and Fujian Key Laboratory of Granular Computing and Applications (Minnan Normal University), China. Huoxiong Wu is supported by National Natural Science Foundation of China (Grant nos. 12171399, 12271041).}
\address{School of Mathematics and Statistics, Minnan Normal University, Zhangzhou 363000, China} \email{wenyongmingxmu@163.com}
\address{School of Mathematical Sciences, Xiamen University, Xiamen 361005, China} \email{huoxwu@xmu.edu.cn}



\begin{abstract}
Let $e^{-tL}$ be a analytic semigroup generated by $-L$, where $L$ is a non-negative self-adjoint operator on $L^2(\mathbb{R}^d)$. Assume that the kernels of $e^{-tL}$, denoted by $p_t(x,y)$, only satisfy the upper bound: for all $N>0$, there are constants $c,C>0$ such that
\begin{align}\label{upper bound}
|p_t(x,y)|\leq\frac{C}{t^{d/2}}e^{-\frac{|x-y|^2}{ct}}\Big(1+\frac{\sqrt{t}}{\rho(x)}+
\frac{\sqrt{t}}{\rho(y)}\Big)^{-N}
\end{align}
holds for all $x,y\in\mathbb{R}^d$ and $t>0$. We first establish the quantitative matrix weighted inequalities for fractional type integrals associated to $L$ with new classes of matrix weights, which are nontrivial extension of the results established by Li, Rahm and Wick \cite{LRW}. Next, we give new two-weight bump conditions with Young functions satisfying wider conditions for fractional type integrals associated to $L$, which cover the result obtained by Cruz-Uribe, Isralowitz and Moen \cite{CrIsMo}. We point out that the new classes of matrix weights and bump conditions are larger and weaker than the classical ones given in \cite{IM} and \cite{CrIsMo}, respectively. As applications, our results can be applied to settings of magnetic Schr\"{o}dinger operator, Laguerre operators, etc.
\end{abstract}

\maketitle

\section{Introduction and main results}
\subsection{Background}

In recent years, the problem of quantitative weighted estimates for singular integrals and related operators has appealed the attention of many mathematicians. Buckley \cite{Buckley} first studied sharp weighted inequalities for Hardy-Littlewood maximal function $M$ as follows:
\begin{align*}
\|Mf\|_{L^p(\omega)}\lesssim[\omega]_{A_p}^{1/(p-1)}\|f\|_{L^p(\omega)},~1<p<\infty,
\end{align*}
where the $A_p$ constant $[\omega]_{A_p}$ is defined by
\begin{equation}\label{Ap}
[\omega]_{A_p}:=\sup_Q\Big(\frac{1}{|Q|}\int_Q\omega(x)dx\Big)
\Big(\frac{1}{|Q|}\int_Q\omega(x)^{1-p'}dx\Big)^{p-1}<\infty.
\end{equation}
Later, Petermichl \cite{P,P1} considered the sharp $A_p$ estimates for Hilbert and Riesz transforms. While for general Calder\'{o}n-Zygmund operators $T$, Hyt\"{o}nen \cite{Hy} proved the dyadic representation theorem for $T$, which leads to the full picture of the sharp $A_p$ bounds for $T$. Subsequently, Lacey, Moen, P\'{e}rez and Torres \cite{LaMPT} established the following sharp weighted estimates for the classical fractional integral operator $I_\alpha$ and fractional maximal operator $M_\alpha$ in terms of the $A_{p,q}$ constant.
\begin{theorem}{\rm(\cite{LaMPT})}\label{quantitative weighted for classical fractional}
Suppose that $0<\alpha<d$, $1<p<d/\alpha$ and $q$ satisfies $1/q=1/p-\alpha/d$. Then
\begin{align*}
\|I_\alpha f\|_{L^q(\omega)}\lesssim[\omega]_{A_{p,q}}^{(1-\alpha/d)\max\{1,p'/q\}}
\|f\|_{L^p(\omega^{p/q})}
\end{align*}
and
\begin{align*}
\|M_\alpha f\|_{L^{q}(\omega)}\lesssim[\omega]_{A_{p,q}}^{(1-\alpha/d)p'/q}
\|f\|_{L^p(\omega^{p/q})}.
\end{align*}
Furthermore, both of two results above are sharp, where
$$I_\alpha f(x):=\int_{\mathbb{R}^d}\frac{f(y)}{|x-y|^{n-\alpha}}dy,~M_\alpha f(x):=\sup_{Q\ni x}\frac{1}{|Q|^{1-\alpha/d}}\int_Q|f(y)|dy$$
and the $A_{p,q}$ constant $[\omega]_{A_{p,q}}$ is defined by
$$[\omega]_{A_{p,q}}:=\sup_{Q}\Big(\frac{1}{|Q|}\int_Q\omega(x)dx\Big)
\Big(\frac{1}{|Q|}\int_Q\omega(x)^{-p'/q}dx\Big)^{q/p'}<\infty.$$
\end{theorem}
Inspired by the above work, Li, Rahm and Wick \cite{LRW} first investigated the quantitative weighted estimates of fractional integral $L^{-{\alpha}/{2}}$ and fractional maximal operator $M_\alpha^{\rho,\theta}$ in the Schr\"{o}dinger setting, where $L^{-{\alpha}/{2}}$ and $M_\alpha^{\rho,\theta}$ are defined as follows.
\begin{align}\label{fractional integral associated with operators}
L^{-{\alpha}/{2}}f(x):=\int_0^\infty e^{-tL}f(x)t^{\frac{\alpha}{2}-1}dt,~0<\alpha<d,
\end{align}
\begin{align*}
M_\alpha^{\rho,\theta}f(x):=\sup_{Q\ni x}\frac{1}{(\psi_\theta(Q)|Q|)^{1-\frac{\alpha}{d}}}\int_Q|f(y)|dy,~\theta\geq0,
\end{align*}
where $\psi_\theta(Q):=(1+l_Q/\rho(x_Q))^\theta$, $\rho$ is the critical radius function (see Section 2 for a precise definition), and $x_Q$, $l_Q$ are the center of cube $Q$ and the side-length of $Q$, respectively. The main results of Li, Rahm and Wick \cite{LRW} are stated as follows.
\begin{theorem}{\rm(\cite{LRW})}\label{quantitative weighted estimates for schrodinger fractional}
Suppose that $0<\alpha<d$. Let $1\leq p<d/\alpha$ and $q$ satisfy $1/q=1/p-\alpha/d$. Let $\theta\geq0$, $\gamma=\theta/\big(1+p'/q\big)$ and $K$ be defined by the equation $(1/K+q/(Kp'))(1-\alpha/d)\max\{1,p'/q\}=1/2$. Then
\begin{align*}
\|L^{-{\alpha}/{2}} f\|_{L^q(\omega)}\lesssim[\omega]_{A_{p,q}^{\rho,\theta/(3K)}}
^{(1-\alpha/d)\max\{1,p'/q\}}
\|f\|_{L^p(\omega^{p/q})},~\omega\in A_{p,q}^{\rho,\theta/(3K)},
\end{align*}
and
\begin{align*}
\|M_\alpha^{\rho,\theta} f\|_{L^{q}(\omega)}\lesssim[\omega]_{A_{p,q}^{\rho,\gamma/3}}^
{(1-\alpha/d)p'/q}
\|f\|_{L^p(\omega^{p/q})},~\omega\in A_{p,q}^{\rho,\gamma/3},
\end{align*}
where $A_{p,q}^{\rho,\theta}$ is the class of weights $\omega$ such that
\begin{equation*}
[\omega]_{A_{p,q}^{\rho,\theta}}:=\sup_Q\Big(\frac{1}{\psi_\theta(Q)|Q|}\int_Q\omega(x)dx\Big)
\Big(\frac{1}{\psi_\theta(Q)|Q|}\int_Q\omega(x)^{-p'/q}dx\Big)^{q/p'}<\infty.
\end{equation*}
\end{theorem}

Recall that a matrix weight $W:\mathbb{R}^d\rightarrow\mathbb{C}^{n\times n}$ is a self-adjoint matrix function with locally integrable entries such that $W(x)$ is almost everywhere positive definite almost everywhere. In the setting of matrix weight, Bickel, Petermichl and Wick \cite{BiPeWi} proved the following quantitative weighted estimates for Hilbert transform $H$,
\begin{equation*}
\|H\|_{L^2(W)\rightarrow L^2(W)}\lesssim[W]_{A_2}^{3/2}\log([W]_{A_2}),
\end{equation*}
where $L^p(W)$ is the collection of measurable vector-valued functions $\vec{f}:\mathbb{R}^d\rightarrow \mathbb{C}^n$ with
\begin{equation*}
\|\vec{f}\|_{L^p(W)}:=\Big(\int_{\mathbb{R}^d}|W(x)^{1/p}\vec{f}(x)|^pdx\Big)^{1/p}<\infty,
\end{equation*}
and $W^r$ for any $r\in\mathbb{R}$ is defined by setting $W^r:=OD(\lambda_i^r)O^T$ for some measurable orthogonal matrix function $O$, here $D(\lambda_i)$ is the diagonal matrix and $\lambda_i$ ($i=1,2,\cdots,n$) are the positive eigenvalues of $W$. The result above was improved by Nazarov et al. \cite{NaPeTrVo} and Culiuc et al. \cite{CuPlOu} by extending it to all Calder\'{o}n-Zygmund operators $T$ and eliminating $\log([W]_{A_2})$. It was also extended to all $p\in(1,\infty)$ by Cruz-Uribe, Isralowitz and Moen \cite{CrIsMo}. However, only few sharp quantitative matrix weighted estimates are known until now, see \cite{HyPeVo,IM,IsPoRi,LeLiOmRi}. Let $W$ be a matrix weight and $1<p,q<\infty$. Isralowitz and Moen \cite{IM} introduced the matrix $\mathcal{A}_{p,q}$ weight, we write $W\in \mathcal{A}_{p,q}$ if
\begin{equation*}
[W]_{\mathcal{A}_{p,q}}:=\sup_{Q}\frac{1}{|Q|}\int_Q\Big(\frac{1}{|Q|}
\int_Q|W(x)^{1/q}W(y)^{-1/q}|_{op}^{p'}dy\Big)^{q/p'}dx<\infty,
\end{equation*}
where
\begin{equation*}
|W(x)|_{op}:=\underset{|\vec{e}|=1}{\sup_{\vec{e}\in\mathbb{C}^n}}|W(x)\vec{e}|.
\end{equation*}
Isralowitz and Moen \cite{IM} proved the following quantitative matrix weighted estimates for $I_\alpha$ and matrix-weighted fractional maximal function $M_{W,\alpha}$.
\begin{theorem}{\rm(\cite{IM})}\label{quantitative matrix weighted for classical fractional}
Suppose $0<\alpha<d$, $1<p<d/\alpha$ and $1/q=1/p-\alpha/d$. If $W\in \mathcal{A}_{p,q}$, then
\begin{equation*}
\|M_{W,\alpha}\|_{L^p\rightarrow L^q}\lesssim[W]_{\mathcal{A}_{p,q}}^{p'(1-\alpha/d)/q},~
\|I_\alpha\|_{L^p(W^{p/q})\rightarrow L^q(W)}\lesssim[W]_{\mathcal{A}_{p,q}}^{p'(1-\alpha/d)/q+1/q'},
\end{equation*}
where
\begin{equation*}
M_{W,\alpha}\vec{f}(x):=\sup_{Q\ni x}\frac{1}{|Q|^{1-\alpha/d}}\int_Q|W(x)^{1/q}W(y)^{-1/q}\vec{f}(y)|dy.
\end{equation*}
\end{theorem}
For more historical works on the matrix weighted estimates and quantitative matrix weighted estimates for harmonic analysis operators, we refer readers to see \cite{DiHyLi,DuLiYa,G,Is,MuRi,T,TV,V}, and references therein.

On the other hand, it is well-known that the $A_p$ weight can be characterized in terms of the $L^p(\omega)$-boundedness of Hardy-Littlewood maximal operator $M$ or Hilbert transform $H$. Muckenhoupt \cite{M} raised the question of what type of condition can be used to characterize $M$ or $H$: $L^p(\nu)\rightarrow L^p(\mu)$. Analogously to the one-weight case, it is very natural to consider the condition
\begin{equation*}
\sup_Q\Big(\frac{1}{|Q|}\int_Q\mu(x)dx\Big)
\Big(\frac{1}{|Q|}\int_Q\nu(x)^{1-p'}dx\Big)^{p-1}<\infty.
\end{equation*}
Unfortunately, this condition is not sufficient for singular integral operators to be bounded from $L^p(\nu)$ to $L^p(\mu)$, see \cite{CRMaPe}. To solve this problem, Sawyer \cite{S} first introduced the testing condition, which is sufficient and necessary for $M$ to be bounded from $L^p(\nu)$ to $L^p(\mu)$. However, due to the operator $M$ itself involving into the testing condition, researchers attempted to seek for some sufficient conditions that are easy to verify and close to \eqref{Ap} in some sense. In 1995, P\'{e}rez \cite{Pe} introduced the ``Orlicz bump'' condition for $M$: $L^p(\nu)\rightarrow L^p(\mu)$ as follows:
\begin{equation*}
\sup_Q\|\mu^{1/p}\|_{p,Q}\|\nu^{-1/p}\|_{\Phi,Q}<\infty,~1<p<\infty,
\end{equation*}
where $\Phi$ is a Young function and $\bar{\Phi}\in B_p$. Here, we interpret some notations. Young function $\Phi:[0,\infty)\rightarrow[0,\infty)$ is a continuous, increasing, convex function that satisfies $\Phi(0)=0$ and $\lim_{t\rightarrow\infty}\Phi(t)/t=\infty$. The corresponding complementary function of $\Phi$, denoted by $\bar{\Phi}$, is given by
\begin{equation*}
\bar{\Phi}(t)=\sup_{s>0}\{st-\Phi(s)\}.
\end{equation*}
We call that $\Phi\in B_p$ if
\begin{equation*}
\int_1^\infty\frac{\Phi(t)}{t^p}\frac{dt}{t}<\infty,~1<p<\infty.
\end{equation*}
Let $\Phi$ be a Young function. The localized Orlicz norm $\|f\|_{\Phi,Q}$ is defined by
\begin{equation*}
\|f\|_{\Phi,Q}:=\inf\Big\{\lambda>0:\frac{1}{|Q|}\int_Q\Phi\Big(\frac{|f(x)|}{\lambda}\Big)dx\leq1\Big\}.
\end{equation*}
Cruz-Uribe and P\'{e}rez \cite{CrPe} conjectured that if a pair of weights $(\mu,\nu)$ satisfies
\begin{equation*}
\sup_Q\|\mu^{1/p}\|_{\Psi,Q}\|\nu^{-1/p}\|_{\Phi,Q}<\infty,~1<p<\infty,
\end{equation*}
with $\bar{\Phi}\in B_{p}$ and $\bar{\Psi}\in B_{p'}$, then the Calder\'{o}n-Zygmund operators $T$ is bounded from $L^p(\nu)$ to $L^p(\mu)$. This conjecture was finally solved by Lerner \cite{Lerner}. For fractional integrals, P\'{e}rez \cite{Pe1} proved the following Orlicz bump condition
\begin{equation*}
\sup_Q|Q|^{\alpha/d+1/q-1/p}\|\mu^{1/q}\|_{\Psi,Q}\|\nu^{-1/p}\|_{\Phi,Q}<\infty,~\bar{\Psi}\in B_p,~\bar{\Phi}\in B_{q'}
\end{equation*}
is sufficient for $I_\alpha: L^p(\nu)\rightarrow L^q(\mu)$. Cruz-Uribe and Moen \cite{CrMo} improved the conditions in \cite{Pe1} to weaker $B_{p,q}$ conditions: $\bar{\Psi}\in B_{p,q}$, $\bar{\Phi}\in B_{q',p'}$, where $\bar{\Psi}\in B_{p,q}$ if
\begin{equation*}
\int_1^\infty\frac{\bar{\Psi}(t)^{q/p}}{t^q}\frac{dt}{t}<\infty.
\end{equation*}
In 2018, Cruz-Uribe, Isralowitz and Moen \cite{CrIsMo} extended theory of two weight, Orlicz bump conditions to the setting of matrix weight. For the intension of this paper, we only state their results concerned with fractional type integrals.
\begin{theorem}{\rm(\cite{CrIsMo})}\label{matrix weight bump condition classical fractional maximal}
Let $0\leq\alpha<d$, $1<p\leq q<\infty$ with $1/p-1/q\leq\alpha/d$. Assume that $\Phi$ is a Young function satisfying $\bar{\Phi}\in B_{p,q}$. If a pair $(U,V)$ of matrix weights satisfies
\begin{equation*}
\sup_Q|Q|^{\alpha/d+1/q-1/p}\Big(-\!\!\!\!\!\!\!\int_Q
\|U(x)^{1/q}V^{-1/q}\|_{\Phi,Q}^qdx\Big)^{1/q}<\infty.
\end{equation*}
Then $M_{U,V,\alpha}: L^p(\mathbb{R}^d,\mathbb{C}^n)\rightarrow L^q(\mathbb{R}^d,\mathbb{C}^n)$, where
\begin{equation*}
M_{U,V,\alpha}\vec{f}(x):=\sup_{Q\ni x}\frac{1}{|Q|^{1-\alpha/d}}\int_Q|U(x)^{1/q}V(y)^{-1/q}\vec{f}(y)|dy.
\end{equation*}
\end{theorem}

\begin{theorem}{\rm(\cite{CrIsMo})}\label{matrix weight bump condition classical fractional}
Let $0<\alpha<d$, $1<p\leq q<\infty$ with $1/p-1/q\leq\alpha/d$. Assume that $\Phi$, $\Psi$ are Young function satisfying $\bar{\Phi}\in B_{p,q}$ and $\bar{\Psi}\in B_{q'}$. If a pair $(U,V)$ of matrix weights satisfies
\begin{equation*}
\sup_Q|Q|^{\alpha/d+1/q-1/p}\Big\|\|U(x)^{1/q}V(y)^{-1/q}\|_{\Phi_y,Q}\Big\|_{\Psi_x,Q}<\infty.
\end{equation*}
Then $I_\alpha: L^p(V^{p/q})\rightarrow L^q(U)$.
\end{theorem}
\begin{remark}\label{conjecture}
In the scalar case, since Theorem \ref{matrix weight bump condition classical fractional} does not recapture the weaker hypothesis $\bar{\Psi}\in B_{q',p'}$ in \cite{CrMo}, the authors in \cite{CrIsMo} conjectured that Theorem \ref{matrix weight bump condition classical fractional} remains true if $\bar{\Psi}\in B_{q'}$ is replaced by $\bar{\Psi}\in B_{q',p'}$.
\end{remark}

\subsection{Aims and questions}
Inspired by the works in \cite{CrIsMo,IM,LRW}, this paper is devoted to studying the quantitative matrix weighted estimates and two matrix weight bounds for fractional type integrals associated with general differential operators. To be more precise, let $L$ be a nonnegative self-adjoint operator defined on $L^2(\mathbb{R}^d)$. Thus, $L$ generates an analytic heat semigroup $e^{-tL}$. Denote the kernel of $e^{-tL}$ by $p_t(x,y)$, which merely satisfies \eqref{upper bound}. Let $\theta\geq0$, $U$, $V$ be matrix weights, $\mathcal{U}_{Q}^{q,q}$ be reducing operator such that
\begin{equation*}
|\mathcal{U}_{Q}^{q,q}\vec{e}|\sim\|U^{1/q}\vec{e}\|_{q,Q},~1<q<\infty.
\end{equation*}
We will consider the fractional integrals $L^{-{\alpha}/{2}}$ defined as in \eqref{fractional integral associated with operators}, matrix-weighted fractional maximal function associated with critical radius function $M_{U,V,\psi,\alpha}$, auxiliary matrix-weighted fractional maximal function associated with critical radius function $M_{\mathcal{U},V,\psi,\alpha}$, averaging operator associated with critical radius function $A_{Q}^{\alpha,\theta}$, auxiliary averaging operator associated with critical radius function $\mathcal{A}_{Q}^{\alpha,\theta}$, which are defined by
\begin{equation*}
M_{U,V,\psi_\theta,\alpha}\vec{f}(x):=\sup_{Q\ni x}\frac{1}{(\psi_\theta(Q)|Q|)^{1-\alpha/d}}\int_Q|U(x)^{1/q}V(y)^{-1/q}\vec{f}(y)|dy,
\end{equation*}
\begin{equation*}
M_{\mathcal{U},V,\psi_\theta,\alpha}\vec{f}(x):=\sup_{Q\ni x}
\frac{1}{(\psi_\theta(Q)|Q|)^{1-\alpha/d}}
\int_{Q}|\mathcal{U}_{Q}^{q,q}V(y)^{-1/q}\vec{f}(y)|dy,
\end{equation*}
\begin{equation*}
A_{Q}^{\alpha,\theta}\vec{f}(x):=\frac{1}{(\psi_\theta(Q)|Q|)^{1-\alpha/d}}
\int_{Q}\vec{f}(y)dy\chi_Q(x),
\end{equation*}
and
\begin{equation*}
\mathcal{A}_{Q}^{\alpha,\theta}\vec{f}(x):=\frac{1}{(\psi_\theta(Q)|Q|)^{1-\alpha/d}}
\int_{Q}|\mathcal{U}_{Q}^{q,q}V(y)^{-1/q}\vec{f}(y)|dy\chi_Q(x),
\end{equation*}
respectively. If $U=V=W$, we simply write $M_{U,V,\psi_\theta,\alpha}$ by $M_{W,\psi_\theta,\alpha}$.

To obtain the quantitative matrix weighted estimates for fractional type integrals associated with $L$, a natural ideal is to adopt the idea in \cite{BuiBD,BuiBD1} to dominate the local part by sparse operator, and dominate the global part by maximal operator, however, compare to the scalar case, matrix products of self-adjoint matrices do not commute, moreover, the technique to compare objects and dominate one by another is lost in the vector case, this prevent us from utilizing the technique in \cite{BuiBD,BuiBD1}. Besides, there is additional critical radius function factor in the operator. Hence, the following question is natural.

\textbf{Question 1:} How to obtain the quantitative matrix weighted estimates for fractional type integrals? Moreover, replace the Schr\"{o}dinger operator by more general differential operator $L$, we may face new challenges since the kernels of the semigroup generated by $-L$ are not assumed to satisfy any regularity conditions.

Next, the upcoming issues is the classes of matrix weights that we deal with. In light of the classes of weights in Theorem \ref{quantitative weighted estimates for schrodinger fractional}, it is likely that new classes of matrix weights will be needed.

\textbf{Question 2:} In the setting of $L$, what types of matrix weights and bump conditions are appropriate for quantitative matrix weighted estimates and two-weight inequalities of fractional type integrals, respectively? If new classes of weights exist, how to deal with these classes of matrix weights to achieve the desired conclusions? Can we find some characterizations of these classes of matrix weights?

Finally, concerned with Remark \ref{conjecture}, we also conjecture this conclusion is still true for fractional integrals associated with $L$. However, we believe that there is still a long way to prove this conjecture.

\textbf{Question 3:} In our new setting, can we take a step forward on the path of proving this conjecture?

\subsection{Main results}
Our first results are concerned with the quantitative matrix weighted estimates for fractional type integrals. The definitions of new classes of matrix weights are given in Section 2.
\begin{theorem}\label{quantitative matrix weighted estimates for fractional integrals associated with operators}
Let $\theta\geq0$, $0<\alpha<d$, $1<p<d/\alpha$ and $1/q=1/p-\alpha/d$. Let $K$ satisfy $(\frac{1}{K}+\frac{q}{Kp'})[(1-\frac{\alpha}{d})\frac{p'}{q}+\frac{1}{q'}]=\frac{1}{2}$. If $W\in \mathcal{A}_{p,q}^{\rho,\theta/(3K)}$, then
\begin{equation*}
\|L^{-{\alpha}/{2}}\|_{L^p(W^{p/q})\rightarrow L^q(W)}\lesssim[W]_{\mathcal{A}_{p,q}^{\rho,\theta/(3K)}}^{\frac{p'}{q}(1-\frac{\alpha}{d})+\frac{1}{q'}}.
\end{equation*}
\end{theorem}

\begin{theorem}\label{quantitative matrix weighted estimates for fractional maximal operator associated with operators}
Let $\theta\geq0$, $0\leq\alpha<d$, $1<p<d/\alpha$ and $1/q=1/p-\alpha/d$.\\
{\rm (1)} If $M_{W,\psi,\alpha}$ is bounded from $L^p$ to $L^q$, then $W\in \mathcal{A}_{p,q}^{\rho,\theta}$.\\
{\rm (2)} Let $K$ be defined by the equation $p'/(Kq)+1/K=1/2$. If $W\in \mathcal{A}_{p,q}^{\rho,\theta/(3K)}$, then
\begin{equation*}
\|M_{W,\psi_\theta,\alpha}\|_{L^p\rightarrow L^q}\lesssim[W]_{\mathcal{A}_{p,q}^{\rho,\theta/(3 K)}}^{p'(1-\alpha/d)/q}.
\end{equation*}
\end{theorem}

Next, we present two-weight inequalities for fractional type integrals associated to $L$. Notice that the conditions of Young functions in Theorem \ref{two weight inequality for fractional integrals associated with operators} cover the ones in Theorem \ref{matrix weight bump condition classical fractional} since $B_{q',q'}= B_{q'}$ when $s=q$.
\begin{theorem}\label{two weight inequality for fractional integrals associated with operators}
Let $0<\alpha<d$, $p<q$, $1<p\leq s\leq q<\infty$. Assume that $\Phi$ and $\Psi$ are Young functions such that $\bar{\Phi}\in B_{p,s}$ and $\bar{\Psi}\in B_{q',s'}$. If a pair of weights $(U,V)$ satisfies
\begin{equation*}
[U,V]_{p,q,\alpha,\psi_\theta,\Phi,\Psi}:=\sup_Q
\frac{|Q|^{\alpha/d+1/q-1/p}}{\psi_{\theta/3}(Q)}
\Big\|\|U(x)^{1/q}V(y)^{-1/q}\|_{\Phi_y,Q}\Big\|_{\Psi_x,Q}<\infty,
\end{equation*}
where $\|\cdot\|_{\Phi_y,Q}$ means the Orlicz norm is taken with respect to the $y$ variable and $\|U\|$ means $\||U|_{op}\|$, then
\begin{equation*}
\|L^{-{\alpha}/{2}}\vec{f}\|_{L^q(U)}\lesssim
[U,V]_{p,q,\alpha,\psi_\theta,\Phi,\Psi}\|\vec{f}\|_{L^p(V^{p/q})}.
\end{equation*}
\end{theorem}

\begin{theorem}\label{two weight inequality for fractional maximal operators associated with operators}
Let $\theta>0$, $0\leq\alpha<d$, $1<p\leq q<\infty$ and $1/p-1/q\leq\alpha/d$. Assume that $\bar{\Phi}\in B_{p,q}$. If a pair $(U,V)$ of matrix weights satisfies
\begin{equation*}
[U,V]_{p,q,\alpha,\psi_{\theta'},\Phi}:=\sup_{Q}\frac{|Q|^{\alpha/d+1/q-1/p}}
{\psi_{\theta'/3}(Q)^{1-\alpha/d}}
\Big(-\!\!\!\!\!\!\!\int_Q\|U(x)^{1/q}V^{-1/q}\|_{\Phi,Q}^qdx\Big)^{1/q}<\infty,~0<\theta'<\theta,
\end{equation*}
then $\|M_{U,V,\psi_\theta,\alpha}\vec{f}\|_{L^q}\lesssim[U,V]_{p,q,\alpha,\psi_{\theta'},\Phi}
\|\vec{f}\|_{L^p}$.
\end{theorem}

\begin{remark}
It is obvious that the class of matrix weight $\mathcal{A}_{p,q}^{\rho}$ are larger than $\mathcal{A}_{p,q}$. Besides, the bump conditions are weaker than the ones in Theorems \ref{matrix weight bump condition classical fractional maximal}, \ref{matrix weight bump condition classical fractional}. Hence, our results give better quantitative constant for fractional type integrals.
\end{remark}

\begin{theorem}\label{boundedness auxiliary fractional maximal operator}
Let $\theta\geq0$, $0\leq\alpha<d$, $1\leq p\leq q<\infty$ with $1/p-1/q=\alpha/d$. Let $(U,V)$ be a pair of matrix weights.
\medskip

{\rm (1)} If $(U,V)\in \mathcal{A}_{p,q}^{\rho,\theta/3}$, then
$$\max\{\|M_{\mathcal{U},V,\psi_\theta,\alpha}\vec{f}\|_{L^{q,\infty}}, \|\mathcal{A}_{Q}^{\alpha,\theta}\vec{f}\|_{L^{q,\infty}}\}\lesssim
[U,V]_{\mathcal{A}_{p,q}^{\rho,\theta/3}}^{1/q}\|\vec{f}\|_{L^p},$$
where the implicit constant is independent of $Q$;

{\rm (2)} If $M_{\mathcal{U},V,\psi_\theta,\alpha}$ or $\mathcal{A}_{Q}^{\alpha,\theta}$ is bounded from $L^p$ to $L^{q,\infty}$ with norm independent of $Q$, then $(U,V)\in \mathcal{A}_{p,q}^{\rho,\theta}$.

\medskip
\end{theorem}

\begin{remark}
The original purpose of Theorem \ref{quantitative matrix weighted estimates for fractional maximal operator associated with operators} and Theorem \ref{boundedness auxiliary fractional maximal operator} is to give a characterization of $\mathcal{A}_{p,q}^{\rho,\theta}$ in terms of the boundedness of $M_{W,\psi_\theta,\alpha}$ and $M_{\mathcal{U},V,\psi_\theta,\alpha}$, respectively, as in \cite{CrIsMo,IM}. However, it is out of our expectations that we can not do it due to the annoying critical radius function in the operators.
\end{remark}

The last theorem is related to the characterization of matrix weights.
\begin{theorem}\label{character weight via average operator}
Let $\theta\geq0$, $0\leq\alpha<d$, $1\leq p\leq q<\infty$ such that $1/p-1/q=\alpha/d$. Let $(U,V)$ be a pair of matrix weights. The following statements are equivalent:
\medskip

{\rm (1)} $(U,V)\in \mathcal{A}_{p,q}^{\rho,\theta}$;

\medskip
{\rm (2)} $\|A_{Q}^{\alpha,\theta}\vec{f}\|_{L^q(U)}\lesssim\|\vec{f}\|_{L^{p}(V^{p/q})}$ with norm independent of $Q$.\\
Moreover, we have the estimate
\begin{equation*}
\|A_{Q}^{\alpha,\theta}\vec{f}\|_{L^q(U)}\lesssim
[U,V]_{\mathcal{A}_{p,q}^{\rho,\theta}}^{1/q}\|\vec{f}\|_{L^{p}(V^{p/q})}.
\end{equation*}
\end{theorem}

\begin{remark}
Our results have applications in the settings of magnetic Schr\"{o}dinger, Laguerre operator et al, see \cite{BuiBD}.
\end{remark}

The structure of the paper is as follows. In Section 2, we introduce some new classes of matrix weights and some properties of them. In Section 3, we prove Theorems \ref{quantitative matrix weighted estimates for fractional integrals associated with operators}, \ref{quantitative matrix weighted estimates for fractional maximal operator associated with operators}. The proofs of Theorems \ref{two weight inequality for fractional integrals associated with operators}, \ref{two weight inequality for fractional maximal operators associated with operators}, \ref{boundedness auxiliary fractional maximal operator} will be given in Section 4. Finally, we give the proof of Theorem \ref{character weight via average operator} in Section 5.

Throughout the rest of the paper, we denote $f\lesssim g$, $f\thicksim g$ if $f\leq Cg$ and $f\lesssim g \lesssim f$, respectively. For any cube $Q:=Q(x_Q,r_Q)\subset \mathbb{R}^d$, denote the side-length of $Q$ by $l_Q$, $x_Q$ represents the center of $Q$ and $r_Q=l_Q/2$. $\chi_Q$ means the characteristic function of $Q$ and $\sigma Q:=Q(x_Q,\sigma r_Q)$.

\section{New classes of matrix weights associated to critical radius functions}
In this section, we introduce some new classes of matrix weights associated to critical radius functions. Recall that a function $\rho:\mathbb{R}^d\rightarrow(0,\infty)$ is called critical radius function if there are positive constants $C_0$ and $N_0$ such that for arbitrary $x,y\in\mathbb{R}^d$,
\begin{align}\label{critical radius function}
C_0^{-1}\rho(x)\Big(1+\frac{|x-y|}{\rho(x)}\Big)^{-N_0}\leq\rho(y)\leq C_0\rho(x)
\Big(1+\frac{|x-y|}{\rho(x)}\Big)^{\frac{N_0}{N_0+1}}.
\end{align}
\begin{definition}\label{two matrix weight}
Let $\theta\geq0$, $1\leq p\leq q<\infty$ and $\rho$ be a critical radius function. We say that a pair of matrix weights $(U,V)\in \mathcal{A}_{p,q}^{\rho,\theta}$ with $1<p\leq q<\infty$ if
\begin{equation*}
[U,V]_{\mathcal{A}_{p,q}^{\rho,\theta}}:=\sup_{Q}\frac{1}{\psi_\theta(Q)|Q|}\int_Q
\Big(\frac{1}{\psi_\theta(Q)|Q|}\int_Q|U(x)^{1/q}V(y)^{-1/q}|_{op}^{p'}dy\Big)^{q/p'}dx<\infty.
\end{equation*}
When $p=1$ and $1\leq q<\infty$, we say $(U,V)\in \mathcal{A}_{1,q}^{\rho,\theta}$ if
\begin{equation*}
[U,V]_{\mathcal{A}_{1,q}^{\rho,\theta}}:=\sup_Q\esss_{y\in Q}\frac{1}{\psi_\theta(Q)|Q|}\int_Q
|U(x)^{1/q}V(y)^{-1/q}|_{op}^qdx<\infty.
\end{equation*}
\end{definition}
In particular, if $U=V=W$, we write $(U,V)\in \mathcal{A}_{p,q}^{\rho,\theta}$ by $W\in \mathcal{A}_{p,q}^{\rho,\theta}$. To give some properties of this new class of matrix weights, we recall two very important lemmas related to reducing operators.
\begin{lemma}{\rm(\cite{G})}
Given a matrix weight $A$, a Young function $\Psi$, and a cube $Q$, there exists a matrix $\mathcal{A}_{Q}^\Psi$, called a reducing operator of $A$, such that for all $\vec{e}\in \mathbb{C}^n$,
\begin{equation*}
|\mathcal{A}_{Q}^\Psi\vec{e}|\sim\|A\vec{e}\|_{\Psi,Q}.
\end{equation*}
\end{lemma}

\begin{lemma}{\rm(\cite{CrIsMo})}\label{reducing operator}
Given matrix weights $A$ and $B$, Young functions $\Phi$ and $\Psi$, a cube $Q$, and reducing operators $\mathcal{A}_Q^\Psi$ and $\mathcal{B}_Q^\Phi$, then for all $\vec{e}\in\mathbb{C}^n$,
\begin{equation*}
|\mathcal{A}_Q^\Psi\mathcal{B}_Q^\Phi|_{op}\sim\|A(x)\mathcal{B}_Q^\Phi\|_{\Psi_x,Q}\sim
\Big\|\|A(x)B(y)\|_{\Phi_y,Q}\Big\|_{\Psi_x,Q}.
\end{equation*}
\end{lemma}
Let $\mathcal{W}_{Q}^{q,q}$ and $\widetilde{\mathcal{W}}_{Q}^{q,p'}$ be the reducing operators such that
\begin{equation*}
|\mathcal{W}_{Q}^{q,q}\vec{e}|\sim\|W^{1/q}\vec{e}\|_{q,Q},~
|\widetilde{\mathcal{W}}_{Q}^{q,p'}\vec{e}|\sim\|W^{-1/q}\vec{e}\|_{p',Q}.
\end{equation*}
Based on Lemma \ref{reducing operator}, we have
\begin{equation}\label{equivalent1}
\Big(\frac{1}{|Q|}\int_Q\Big(\frac{1}{|Q|}\int_Q
|W(x)^{1/q}W(y)^{-1/q}|_{op}^{p'}dy\Big)^{q/p'}dx\Big)^{p'/q}\sim
|\mathcal{W}_{Q}^{q,q}\widetilde{\mathcal{W}}_Q^{q,p'}|_{op}^{p'}
\end{equation}
and
\begin{equation}\label{equivalent2}
\frac{1}{|Q|}\int_Q\Big(\frac{1}{|Q|}\int_Q
|W(x)^{-1/q}W(y)^{1/q}|_{op}^{q}dy\Big)^{p'/q}dx\sim
|\widetilde{\mathcal{W}}_Q^{q,p'}\mathcal{W}_{Q}^{q,q}|_{op}^{p'}.
\end{equation}

\begin{theorem}\label{inclusion}
Let $\theta\geq0$, $1\leq p\leq q<\infty$ and $\mathcal{A}_{p,q}^{\rho,\infty}:=\bigcup_{\theta>0}\mathcal{A}_{p,q}^{\rho,\theta}$. Then $\mathcal{A}_{p,q}\subsetneq\mathcal{A}_{p,q}^{\rho}$.
\end{theorem}
\begin{proof}
It is easy to see that $[W]_{\mathcal{A}_{p,q}^{\rho,\theta}}\leq[W]_{\mathcal{A}_{p,q}}$, so $\mathcal{A}_{p,q}\subset\mathcal{A}_{p,q}^{\rho,\infty}$. In the scalar case, the classes $\mathcal{A}_{p,q}^{\rho,\theta}$ and $\mathcal{A}_{p,q}$ reduce to $A_{p,q}^{\rho,\theta}$ and $A_{p,q}$, respectively. It is known that $A_{p,q}\subsetneq A_{p,q}^{\rho,\infty}$. This proves the conclusion.
\end{proof}

\begin{theorem}
Let $\theta\geq0$ and $1<p\leq q<\infty$. Then $[W^{-p'/q}]_{\mathcal{A}_{q',p'}^{\rho,\theta}}\sim[W]_{\mathcal{A}_{p,q}^{\rho,\theta}}^{p'/q}$.
\end{theorem}
\begin{proof}
According to \eqref{equivalent1} and \eqref{equivalent2}, we get
\begin{align*}
[W]_{\mathcal{A}_{p,q}^{\rho,\theta}}^{p'/q}&=\sup_Q
\Big(\frac{1}{\psi_\theta(Q)|Q|}\int_Q\Big(\frac{1}{\psi_\theta(Q)|Q|}\int_Q
|W(x)^{1/q}W(y)^{-1/q}|_{op}^{p'}dy\Big)^{q/p'}dx\Big)^{p'/q}\\
&\sim\sup_Q\psi_\theta(Q)^{-p'/q-1}|\mathcal{W}_{Q}^{q,q}\widetilde{\mathcal{W}}_Q^{q,p'}|_{op}^{p'}
\end{align*}
and
\begin{align*}
[W^{-p'/q}]_{\mathcal{A}_{q',p'}^{\rho,\theta}}&=\sup_Q
\frac{1}{\psi_\theta(Q)|Q|}\int_Q\Big(\frac{1}{\psi_\theta(Q)|Q|}\int_Q
|W(x)^{-1/q}W(y)^{1/q}|_{op}^{q}dy\Big)^{p'/q}dx\\
&=\sup_Q\psi_\theta(Q)^{-p'/q-1}|\widetilde{\mathcal{W}}_Q^{q,p'}\mathcal{W}_{Q}^{q,q}|_{op}^{p'}.
\end{align*}
The conclusion follows immediately by using $|AB|_{op}=|BA|_{op}$ for any two positive definite matrices $A$ and $B$.
\end{proof}

\begin{theorem}
Let $\theta\geq0$, $1\leq p\leq q<\infty$ such that $1/p-1/q=\alpha/d$. If $W\in \mathcal{A}_{p,q}^{\rho,\theta}$, then for any unit vector $\vec{e}$, $|W^{1/q}\vec{e}|^q\in A_{p,q}^{\rho,\theta}$ and
\begin{align*}
[|W^{1/q}\vec{e}|^q]_{A_{p,q}^{\rho,\theta}}\lesssim[W]_{\mathcal{A}_{p,q}^{\rho,\theta}}.
\end{align*}
\end{theorem}
\begin{proof}
Let $\phi$ be arbitrary scalar function and $\vec{f}=\phi\vec{e}$. In the one weight case, Theorem \ref{character weight via average operator} indicates that
\begin{equation*}
\phi\mapsto\frac{\chi_Q}{(\psi_\theta(Q)|Q|)^{1-\alpha/d}}\int_Q\phi(x)dx
\end{equation*}
is bounded from the scalar weighted $L^p(|W^{1/q}\vec{e}|^p)$ to the scalar weighted $L^q(|W^{1/q}\vec{e}|^q)$ and $[|W^{1/q}\vec{e}|^q]_{A_{p,q}^{\rho,\theta}}\lesssim[W]_{\mathcal{A}_{p,q}^{\rho,\theta}}$.
\end{proof}

We also need the following class of matrix weights, which is an extension of that in \cite{LRW} to the setting of matrix weight.
\begin{definition}\label{two matrix weight2}
Let $\theta\geq0$, $1\leq p\leq q<\infty$ and $\rho$ be a critical radius function. We say that $W\in \widetilde{\mathcal{A}}_{p,q}^{\rho,\theta}$ with $1<p\leq q<\infty$ if
\begin{equation*}
[W]_{\widetilde{\mathcal{A}}_{p,q}^{\rho,\theta}}:=\sup_{Q}\frac{1}{\widetilde{\psi}_\theta(Q)|Q|}
\int_Q\Big(\frac{1}{\widetilde{\psi}_\theta(Q)|Q|}
\int_Q|W(x)^{1/q}W(y)^{-1/q}|_{op}^{p'}dy\Big)^{q/p'}dx<\infty,
\end{equation*}
When $p=1$ and $1\leq q<\infty$, we say $W\in \widetilde{\mathcal{A}}_{1,q}^{\rho,\theta}$ if
\begin{equation*}
[W]_{\widetilde{\mathcal{A}}_{1,q}^{\rho,\theta}}:=\sup_Q\esss_{y\in Q}\frac{1}{\widetilde{\psi}_\theta(Q)|Q|}\int_Q
|W(x)^{1/q}W(y)^{-1/q}|_{op}^qdx<\infty,
\end{equation*}
where
\begin{equation*}
\tilde{\rho}(Q):=\sup_{x\in Q}\rho(x),~\widetilde{\psi}_\theta(Q):=\Big(1+\frac{l_Q}{\tilde{\rho}(Q)}\Big)^\theta.
\end{equation*}
\end{definition}
Denote $\widetilde{\mathcal{A}}_{p,q}^{\rho,\infty}:=
\bigcup_{\theta>0}\widetilde{\mathcal{A}}_{p,q}^{\rho,\theta}$. The following lemma tells us the relationship between $\mathcal{A}_{p,q}^{\rho,\infty}$ and $\widetilde{\mathcal{A}}_{p,q}^{\rho,\infty}$.
\begin{lemma}\label{realtionship}
Let $\theta>0$ and $1\leq p\leq q<\infty$. Then $[W]_{\mathcal{A}_{p,q}^{\rho,\theta}}\leq[W]_{\widetilde{\mathcal{A}}_{p,q}^{\rho,\theta}}$, $[W]_{\widetilde{\mathcal{A}}_{p,q}^{\rho,3\theta}}\lesssim[W]_{\mathcal{A}_{p,q}^{\rho,\theta}}$ and
$\widetilde{\mathcal{A}}_{p,q}^{\rho,\infty}=\mathcal{A}_{p,q}^{\rho,\infty}$.
\end{lemma}
\begin{proof}
Observe that $\widetilde{\psi}_\theta(Q)\leq\psi_\theta(Q)$, so $[W]_{\widetilde{\mathcal{A}}_{p,q}^{\rho,\theta}}\geq[W]_{\mathcal{A}_{p,q}^{\rho,\theta}}$. On the other hand, it was proved in \cite{LRW} that
\begin{equation}\label{rao dong}
\frac{1}{\widetilde{\psi}_{3\theta}(Q)}\lesssim\frac{1}{\psi_\theta(Q)},
\end{equation}
which implies that $[W]_{\widetilde{\mathcal{A}}_{p,q}^{\rho,3\theta}}\lesssim[W]_{\mathcal{A}_{p,q}^{\rho,\theta}}$. Thus, $\widetilde{\mathcal{A}}_{p,q}^{\rho,\infty}=\mathcal{A}_{p,q}^{\rho,\infty}$.
\end{proof}

\section{Quantitative weighted estimates for fractional type integrals associated with operators}
This section is devoted to proving Theorems \ref{quantitative matrix weighted estimates for fractional integrals associated with operators}, \ref{quantitative matrix weighted estimates for fractional maximal operator associated with operators}. Before this, we prove an upper bound for the kernel of $L^{-{\alpha}/{2}}$.
\begin{lemma}\label{another upper bound}
Let $0<\alpha<d$, $K_t(x,y)$ be the kernel of $L^{-{\alpha}/{2}}$ and $\rho$ be critical radius function. Then for every $x,y\in\mathbb{R}^d$ and $N>0$, there is a positive constant $C$ such that
\begin{equation*}
|K_t(x,y)|\leq\frac{C}{\Big(1+|x-y|\Big(\frac{1}{\rho(x)}+\frac{1}{\rho(y)}\Big)\Big)^N}
\frac{1}{|x-y|^{d-\alpha}}.
\end{equation*}
\end{lemma}
\begin{proof}
It is direct that
\begin{align}\label{zong kernel}
|K_t(x,y)|&\leq\int_0^\infty|p_t(x,y)|t^{\frac{\alpha}{2}-1}dt\\
&=\int_0^{|x-y|^2}|p_t(x,y)|t^{\frac{\alpha}{2}-1}dt+\int_{|x-y|^2}^\infty|p_t(x,y)|
t^{\frac{\alpha}{2}-1}dt\nonumber
\end{align}

Applying the upper estimate of $p_t(x,y)$ and $e^{-s}\lesssim s^{-M/2}$ for any $M>0$, we have
\begin{align}\label{local estimate of kernel}
&\int_0^{|x-y|^2}|p_t(x,y)|t^{\frac{\alpha}{2}-1}dt\\
&\quad\lesssim\int_0^{|x-y|^2}t^{\frac{\alpha}{2}-\frac{d}{2}-1}e^{-\frac{|x-y|^2}{ct}}
\Big(1+\frac{\sqrt{t}}{\rho(x)}+\frac{\sqrt{t}}{\rho(y)}\Big)^{-N}dt\nonumber\\
&\quad\lesssim\int_0^{|x-y|^2}t^{\frac{\alpha}{2}-\frac{d}{2}-1}
\frac{t^{\frac{N}{2}+\frac{d}{2}-\frac{\alpha}{2}+1}}{|x-y|^{N+d-\alpha+2}}
\Big(\frac{\sqrt{t}}{|x-y|}+\frac{\sqrt{t}}{\rho(x)}+\frac{\sqrt{t}}{\rho(y)}\Big)^{-N}dt\nonumber\\
&\quad=\frac{1}{\Big(1+|x-y|\Big(\frac{1}{\rho(x)}+\frac{1}{\rho(y)}\Big)\Big)^N}
\frac{1}{|x-y|^{d-\alpha}}.\nonumber
\end{align}

In addition, using the upper bound of $p_t(x,y)$ again, we get
\begin{align*}
&\int_{|x-y|^2}^\infty|p_t(x,y)|t^{\frac{\alpha}{2}-1}dt\\
&\quad\lesssim\int_{|x-y|^2}^\infty t^{\frac{\alpha}{2}-\frac{d}{2}-1}\Big(1+\frac{\sqrt{t}}{\rho(x)}+\frac{\sqrt{t}}{\rho(y)}\Big)^{-N}dt
\nonumber\\
&\quad\leq\frac{1}{\Big(1+|x-y|\Big(\frac{1}{\rho(x)}+\frac{1}{\rho(y)}\Big)\Big)^N}\int_{|x-y|^2}^\infty
t^{\frac{\alpha}{2}-\frac{d}{2}-1}dt\\
&\quad\sim\frac{1}{\Big(1+|x-y|\Big(\frac{1}{\rho(x)}+\frac{1}{\rho(y)}\Big)\Big)^N}
\frac{1}{|x-y|^{d-\alpha}}.
\end{align*}
This together with \eqref{zong kernel} and \eqref{local estimate of kernel}, we prove the conclusion.
\end{proof}

We will prove Theorem \ref{quantitative matrix weighted estimates for fractional integrals associated with operators} via approximating $L^{-{\alpha}/{2}}$ by a dyadic operator. So let us recall some facts about dyadic grid. We call that a collection of cubes $\mathcal{D}$ in $\mathbb{R}^d$ is a dyadic grid if it satisfies:\\
{\rm (1)} for any $Q\in\mathcal{D}$, $l(Q)=2^k$ for some $k\in\mathbb{Z}$;\\
{\rm (2)} for any $Q_1, Q_2\in\mathcal{D}$, $Q_1\cap Q_2=\{Q_1,Q_2,\emptyset\}$;\\
{\rm (3)} for any $k\in\mathbb{Z}$, the family $\mathcal{D}_k=\{Q\in\mathcal{D}:l(Q)=2^k\}$ forms a partition of $\mathbb{R}^d$.

Sparse families is an important sub-family of dyadic grids.
\begin{definition}\rm(\cite{Lerner2017})
Let $\eta\in(0,1)$. A subset $\mathcal{S}\subset\mathcal{D}$ is said to be an $\eta$-sparse family with $\eta\in(0,1)$ if for any cube $Q\in\mathcal{S}$, there exists a measurable subset $E_Q\subset Q$ such that $\eta|Q|\leq|E_Q|$, and the sets $\{E_Q\}_{Q\in\mathcal{S}}$ are mutually disjoint.
\end{definition}

The following lemma shows that any cubes in $\mathbb{R}^d$ can be approximated by cubes from a finite collection of dyadic grids.
\begin{lemma}\rm(\cite{Lerner})\label{dyadic grids}
For $t\in\{0,1/3\}^d$, let $\mathcal{D}^t=\{2^{-k}([0,1)^d+m+(-1)^{k}t):k\in\mathbb{Z},m\in\mathbb{Z}^d\}$. Then each $\mathcal{D}^t$ is a dyadic grid, and given any cube $Q\subset\mathbb{R}^d$, there exist $t$ and $Q^t\in\mathcal{D}^t$ such that $Q\subset Q^t$ and $l_{Q^t}\leq6l_Q$.
\end{lemma}

Based on Lemma \ref{dyadic grids}, we show that $L^{-{\alpha}/{2}}$ is approximated by a dyadic operator.
\begin{proposition}\label{sparse domination for fractional}
Let $\theta\geq0$, $0<\alpha<d$, $1<q<\infty$ and $U,V$ be matrix weights. Then
\begin{align*}
&|\langle U^{1/q}L^{-{\alpha}/{2}}V^{-1/q}\vec{f},\vec{g}\rangle_{L^2}|\\
&\quad\lesssim\sum_{t\in\{0,1/3\}^d}
\sum_{Q\in\mathcal{D}^t}\frac{1}{\widetilde{\psi}_\theta(Q)|Q|^{1-\frac{\alpha}{d}}}\int_Q\int_Q
|\langle V(y)^{-1/q}f(y),U(x)^{1/q}\vec{g}(x)\rangle_{\mathbb{C}^n}|dxdy.
\end{align*}
\end{proposition}
\begin{proof}
Taking into account that $W$ is a self-adjoint matrix and Lemma \ref{another upper bound}, we have
\begin{align*}
&|\langle U^{1/q}L^{-{\alpha}/{2}}V^{-1/q}\vec{f},\vec{g}\rangle_{L^2}|\\
&\quad\lesssim\int_{\mathbb{R}^d}\sum_{k\in\mathbb{Z}}\int_{Q(x,2^k)\backslash Q(x,2^{k-1})}
\frac{|\langle V(y)^{-1/q}\vec{f}(y),U(x)^{1/q}\vec{g}(x)\rangle_{\mathbb{C}^n}|}
{(1+|x-y|/\rho(y))^\theta|x-y|^{d-\alpha}}dydx.
\end{align*}
For $y\in Q(x,2^k)\backslash Q(x,2^{k-1})$, note that $|x-y|\geq2^{k-1}$ and $\rho(y)\leq\tilde{\rho}(Q(x,2^k))$, we get
\begin{align*}
\frac{1}{(1+|x-y|/\rho(y))^\theta|x-y|^{d-\alpha}}\lesssim
\frac{1}{2^{(k-1)(d-\alpha)}\widetilde{\psi}_\theta(Q(x,2^k))}.
\end{align*}
Now for each $k\in \mathbb{Z}$, Lemma \ref{dyadic grids} shows that there are $t\in\{0,1/3\}^d$ and $Q_t\in \mathcal{D}^t$ such that $Q(x,2^k)\subset Q^t$ and
\begin{align*}
2^{k+1}=l_{Q(x,2^k)}\leq l_{Q^t}\leq6l_{Q(x,2^k)}=12\cdot2^k,
\end{align*}
which implies that
\begin{align*}
\frac{2^{-5(d-\alpha)}}{2^{(k-1)(d-\alpha)}}=\frac{1}{2^{(k+4)(d-\alpha)}}\leq
\frac{1}{|Q^t|^{1-\frac{\alpha}{d}}},
\end{align*}
and
\begin{equation*}
\frac{1}{\widetilde{\psi}_\theta(Q(x,2^k))}\lesssim\frac{1}{\widetilde{\psi}_\theta(Q^t)}.
\end{equation*}
Then
\begin{align*}
&|\langle U^{1/q}L^{-{\alpha}/{2}}V^{-1/q}\vec{f},\vec{g}\rangle_{L^2}|\\
&\quad\lesssim\int_{\mathbb{R}^d}
\sum_{k\in\mathbb{Z}}\frac{1}{\widetilde{\psi}_\theta(Q(x,2^k))2^{(k-1)(d-\alpha)}}\int_{Q(x,2^k)}
|\langle V(y)^{-1/q}\vec{f}(y),U(x)^{1/q}\vec{g}(x)\rangle_{\mathbb{C}^n}|dydx\\
&\quad\lesssim\int_{\mathbb{R}^d}
\sum_{k\in\mathbb{Z}}\sum_{t\in\{0,1/3\}^d}\underset{2^{k+1}\leq l_Q\leq2^{k+4}}{\sum_{Q\in\mathcal{D}^t}}\int_{Q(x,2^k)}
|\langle V(y)^{-1/q}\vec{f}(y),U(x)^{1/q}\vec{g}(x)\rangle_{\mathbb{C}^n}|dy\\
&\qquad\times\frac{\chi_Q(x)}{\widetilde{\psi}_\theta(Q)|Q|^{1-\frac{\alpha}{d}}}dx\\
&\quad\lesssim\sum_{t\in\{0,1/3\}^d}\sum_{Q\in\mathcal{D}^t}
\frac{1}{\widetilde{\psi}_\theta(Q)|Q|^{1-\frac{\alpha}{d}}}\int_Q\int_Q
|\langle V(y)^{-1/q}\vec{f}(y),U(x)^{1/q}\vec{g}(x)\rangle_{\mathbb{C}^n}|dxdy.
\end{align*}
\end{proof}

In order to prove Theorem \ref{quantitative matrix weighted estimates for fractional integrals associated with operators}, we define the following class of weights. Let $\mathcal{Q}$ be a collection of dyadic cubes. We say that $W\in \mathcal{A}_{p,q}^{\mathcal{Q}}$ if
\begin{equation*}
[W]_{\mathcal{A}_{p,q}^{\mathcal{Q}}}:=\sup_{Q\in\mathcal{Q}}\frac{1}{|Q|}\int_Q
\Big(\frac{1}{|Q|}\int_Q|W(x)^{1/q}W(y)^{-1/q}|_{op}^{p'}dy\Big)^{q/p'}dx<\infty.
\end{equation*}

\begin{proof}[Proof of Theorem \ref{quantitative matrix weighted estimates for fractional integrals associated with operators}]
Given a matrix weight $W$, observe that $L^{-{\alpha}/{2}}: L^p(W^{p/q})\rightarrow L^q(W)$ if and only if
\begin{equation*}
W^{1/q}L^{-{\alpha}/{2}}W^{-1/q}: L^p(\mathbb{R}^d,\mathbb{C}^n)\rightarrow L^p(\mathbb{R}^d,\mathbb{C}^n).
\end{equation*}
Hence, it suffices to show that
\begin{equation*}
|\langle W^{1/q}L^{-{\alpha}/{2}}W^{-1/q}\vec{f},\vec{g} \rangle_{L^2}|\lesssim[W]_{A_{p,q}^{\rho,\theta/3K}}^{\frac{p'}{q}(1-\frac{\alpha}{d})+\frac{1}{q'}}
\|\vec{f}\|_{L^p}\|\vec{g}\|_{L^{q'}}.
\end{equation*}
By Proposition \ref{sparse domination for fractional}, fix a dyadic grid $\mathcal{D}$, it is enough to prove
\begin{align*}
&\sum_{Q\in\mathcal{D}}\frac{1}{\widetilde{\psi}_\theta(Q)|Q|^{1-\frac{\alpha}{d}}}\int_Q\int_Q
|\langle W(y)^{-1/q}\vec{f}(y),W(x)^{1/q}\vec{g}(x)\rangle_{\mathbb{C}^n}|dxdy\\
&\quad\lesssim[W]_{A_{p,q}^{\rho,\theta/3K}}^{\frac{p'}{q}(1-\frac{\alpha}{d})+\frac{1}{q}}
\|\vec{f}\|_{L^p}\|\vec{g}\|_{L^{q'}}.
\end{align*}

For $r\geq0$. Here and in the following, we denote $\mathcal{Q}_r:=\{Q\in\mathcal{D}:\widetilde{\psi}_\theta(Q)\sim2^{r\theta}\}$. Then
\begin{align*}
&\sum_{Q\in\mathcal{D}}\frac{1}{\widetilde{\psi}_\theta(Q)|Q|^{1-\frac{\alpha}{d}}}\int_Q\int_Q
|\langle W(y)^{-1/q}\vec{f}(y),W(x)^{1/q}\vec{g}(x)\rangle_{\mathbb{C}^n}|dxdy\\
&\quad=\sum_{r\geq0}\sum_{Q\in\mathcal{Q}_r}\frac{1}{\widetilde{\psi}_\theta(Q)|Q|^{1-\frac{\alpha}{d}}}
\int_Q\int_Q|\langle W(y)^{-1/q}\vec{f}(y),W(x)^{1/q}\vec{g}(x)\rangle_{\mathbb{C}^n}|dxdy\\
&\quad\sim\sum_{r\geq0}2^{-r\theta}\sum_{Q\in\mathcal{Q}_r}\frac{1}{|Q|^{1-\frac{\alpha}{d}}}
\int_Q\int_Q|\langle W(y)^{-1/q}\vec{f}(y),W(x)^{1/q}\vec{g}(x)\rangle_{\mathbb{C}^n}|dxdy.
\end{align*}
The operator
\begin{equation*}
\sum_{Q\in\mathcal{Q}_r}\frac{1}{|Q|^{1-\frac{\alpha}{d}}}
\int_Q\int_Q|\langle W(y)^{-1/q}\vec{f}(y),W(x)^{1/q}\vec{g}(x)\rangle_{\mathbb{C}^n}|dxdy
\end{equation*}
is very similar to the following operator given in \cite{IM} except that $\mathcal{Q}_r=\mathcal{D}$
\begin{equation*}
\sum_{Q\in\mathcal{D}}\frac{1}{|Q|^{1-\frac{\alpha}{d}}}
\int_Q\int_Q|\langle W(y)^{-1/q}\vec{f}(y),W(x)^{1/q}\vec{g}(x)\rangle_{\mathbb{C}^n}|dxdy,
\end{equation*}
where the authors use it to obtain the quantitative matrix weighted estimates in Theorem \ref{quantitative matrix weighted for classical fractional}. Now for $W\in\mathcal{A}_{p,q}^{\mathcal{Q}_r}$, by carefully following a similar scheme in the proof of Theorem \ref{quantitative matrix weighted for classical fractional}, we obtain
\begin{align*}
&\sum_{Q\in\mathcal{Q}_r}\frac{1}{|Q|^{1-\frac{\alpha}{d}}}\int_Q\int_Q
|\langle W(y)^{-1/q}\vec{f}(y),W(x)^{1/q}\vec{g}(x)\rangle_{\mathbb{C}^n}|dxdy\\
&\quad\lesssim[W]_{\mathcal{A}_{p,q}^{\mathcal{Q}_r}}^{\frac{p'}{q}(1-\frac{\alpha}{d})+\frac{1}{q'}}
\|\vec{f}\|_{L^p}\|\vec{g}\|_{L^{q'}}.
\end{align*}
Note that for $Q\in\mathcal{Q}_r$,
\begin{align*}
\frac{1}{|Q|}\int_Q\Big(\frac{1}{|Q|}\int_Q|W(x)^{1/q}W(y)^{-1/q}|_{op}^{p'}dy\Big)^{q/p'}dx
\leq[W]_{\tilde{\mathcal{A}}_{p,q}}^{\rho,\theta/K}
2^{r(\theta/K+\theta q/(Kp'))}.
\end{align*}
Thus, $[W]_{\mathcal{A}_{p,q}^{\mathcal{Q}_r}}\leq[W]_{\tilde{\mathcal{A}}_{p,q}^{\rho,\theta/K}}
2^{r(\theta/K+\theta q/(Kp'))}$.
It follows that
\begin{align*}
&\sum_{Q\in\mathcal{D}}\frac{1}{\widetilde{\psi}_\theta(Q)|Q|^{1-\frac{\alpha}{d}}}\int_Q\int_Q
|\langle W(y)^{-1/q}\vec{f}(y),W(x)^{1/q}\vec{g}(x)\rangle_{\mathbb{C}^n}|dxdy\\
&\quad\lesssim\sum_{r\geq0}2^{-r\theta}\Big([W]_{\tilde{\mathcal{A}}_{p,q}^{\rho,\theta/K}}
2^{r(\theta/K+\theta q/(Kp'))}\Big)^{\frac{p'}{q}(1-\frac{\alpha}{d})+\frac{1}{q'}}\|\vec{f}\|_{L^p}\|\vec{g}\|_{L^{q'}}\\
&\quad\lesssim\sum_{r\geq0}2^{-r\theta/2}
[W]_{\mathcal{A}_{p,q}^{\rho,\theta/(3K)}}^{\frac{p'}{q}(1-\frac{\alpha}{d})+\frac{1}{q'}}
\|\vec{f}\|_{L^p}\|\vec{g}\|_{L^{q'}}\\
&\quad\sim[W]_{\mathcal{A}_{p,q}^{\rho,\theta/(3K)}}^{\frac{p'}{q}(1-\frac{\alpha}{d})+\frac{1}{q'}}
\|\vec{f}\|_{L^p}\|\vec{g}\|_{L^{q'}}.
\end{align*}
\end{proof}

\begin{proposition}\label{sparse domination for fractional maximal operator}
Let $\theta\geq0$, $0\leq\alpha<d$, $1<q<\infty$ and $U,V$ be matrix weights. Then for any $x\in\mathbb{R}^d$,
\begin{equation*}
M_{U,V,\psi_\theta,\alpha}\vec{f}(x)\lesssim\sum_{t\in\{0,1/3\}^d}
M_{U,V,\widetilde{\psi}_\theta,\alpha}^{\mathcal{D}^t}\vec{f}(x),
\end{equation*}
where
\begin{equation*}
M_{U,V,\widetilde{\psi}_\theta,\alpha}^{\mathcal{D}}\vec{f}(x):=\underset{Q\ni x}{\sup_{Q\in\mathcal{D}}}
\frac{1}{(\widetilde{\psi}_\theta(Q)|Q|)^{1-\alpha/d}}\int_Q|U(x)^{1/q}V(y)^{-1/q}\vec{f}(y)|dy.
\end{equation*}
\end{proposition}
\begin{proof}
We still employ the notations given in the proof of Proposition \ref{sparse domination for fractional}. For any cube $Q:=Q(x_Q,r_Q)$, by Lemma \ref{dyadic grids}, there exist $t\in\{0,1/3\}^d$ and $Q^t\in\mathcal{D}^t$ such that $Q\subset Q^t$ with $l_{Q^t}\leq12r_Q$. Observe that $\rho(x_Q)\leq\tilde{\rho}(Q^t)$, then
\begin{align*}
&\frac{1}{(\psi_\theta(Q)|Q|)^{1-\alpha/d}}\int_Q|U(x)^{1/q}V(y)^{-1/q}\vec{f}(y)|dy\\
&\quad\lesssim\frac{1}{(\widetilde{\psi}_\theta(Q^t)|Q^t|)^{1-\alpha/d}}\int_{Q^t}
|U(x)^{1/q}V(y)^{-1/q}\vec{f}(y)|dy\\
&\quad\leq M_{U,V,\widetilde{\psi}_\theta,\alpha}^{\mathcal{D}^t}\vec{f}(x).
\end{align*}
Thus,
\begin{equation*}
M_{U,V,\psi_\theta,\alpha}\vec{f}(x)\lesssim\sum_{t\in\{0,1/3\}^d}
M_{U,V,\widetilde{\psi}_\theta,\alpha}^{\mathcal{D}^t}\vec{f}(x).
\end{equation*}
\end{proof}

\begin{proof}[Proof of Theorem \ref{quantitative matrix weighted estimates for fractional maximal operator associated with operators}]
(1). For $x\in\mathbb{R}^d$, fix a cube $Q\ni x$. It is direct that
\begin{align*}
\Big|\frac{\chi_Q(x)}{(\psi_\theta(Q)|Q|)^{1-\alpha/d}}\int_QW(x)^{1/q}\vec{f}(y)dy\Big|&\leq
\frac{\chi_Q(x)}{(\psi_\theta(Q)|Q|)^{1-\alpha/d}}\int_Q\Big|W(x)^{1/q}\vec{f}(y)\Big|dy\\
&\leq M_{W,\psi_\theta,\alpha}(W^{1/q}\vec{f}).
\end{align*}
By the assumption, we have
\begin{equation*}
\sup_Q\Big\|\frac{\chi_Q}{(\psi_\theta(Q)|Q|)^{1-\alpha/d}}\int_Q\vec{f}(y)dy\Big\|_{L^q(W)}
\lesssim\|\vec{f}\|_{L^p(W^{p/q})}.
\end{equation*}
Thus, (1) follows by Theorem \ref{character weight via average operator}, which will be proved in Section 5.
\medskip

(2). By Proposition \ref{sparse domination for fractional maximal operator}, fix a dyadic grid $\mathcal{D}$, it suffices to prove that
\begin{equation*}
\|M_{W,\widetilde{\psi}_\theta,\alpha}^{\mathcal{D}}\vec{f}\|_{L^q}\lesssim
[W]_{A_{p,q}^{\rho,\theta/(3K)}}^{p'(1-\alpha/d)/q}\|\vec{f}\|_{L^p},
\end{equation*}
where $M_{W,\widetilde{\psi}_\theta,\alpha}^{\mathcal{D}}:=M_{W,W,\widetilde{\psi}_\theta,\alpha}^{\mathcal{D}}$. Note that
\begin{equation*}
M_{W,\widetilde{\psi}_\theta,\alpha}^{\mathcal{D}}\vec{f}(x)\leq\sum_{r\geq0}2^{-r\theta(1-\alpha/d)}
\sup_{Q\in\mathcal{Q}_r}\frac{1}{|Q|^{1-\alpha/d}}\int_Q|W(x)^{1/q}W(y)^{-1/q}\vec{f}(y)|dy.
\end{equation*}
Following a similar scheme in the proof of Theorem \ref{quantitative matrix weighted for classical fractional}, we have
\begin{align*}
&\Big\|\sup_{Q\in\mathcal{Q}_r}
\frac{1}{|Q|^{1-\alpha/d}}\int_Q|W(x)^{1/q}W(y)^{-1/q}\vec{f}(y)|dy\Big\|_{L^q}\\
&\quad\lesssim[W]_{\mathcal{A}_{p,q}^{\mathcal{Q}_r}}^{p'(1-\alpha/d)/q}\|\vec{f}\|_{L^p}\\
&\quad\leq([W]_{\tilde{\mathcal{A}}_{p,q}^{\rho,\theta/K}}2^{r(\theta/K+\theta q/(Kp'))})^{p'(1-\alpha/d)/q}\|\vec{f}\|_{L^p}\\
&\quad\lesssim[W]_{\mathcal{A}_{p,q}^{\rho,\theta/(3K)}}^{p'(1-\alpha/d)/q}2^{r\theta(p'/(Kq)+1/K)(1-\alpha/d)}
\|\vec{f}\|_{L^p}.
\end{align*}
Hence,
\begin{align*}
\|M_{W,\psi_\theta,\alpha}^{\mathcal{D}}\vec{f}\|_{L^q}&\lesssim
[W]_{\mathcal{A}_{p,q}^{\rho,\theta/(3K)}}^{p'(1-\alpha/d)/q}\|\vec{f}\|_{L^p}
\sum_{r\geq0}2^{-r\theta(1-\alpha/d)/2}\leq
[W]_{\mathcal{A}_{p,q}^{\rho,\theta/(3K)}}^{p'(1-\alpha/d)/q}\|\vec{f}\|_{L^p}.
\end{align*}
\end{proof}

\section{two-weight estimates for fractional type integrals associated with operators}
The goal of this section is to prove Theorem \ref{two weight inequality for fractional integrals associated with operators}, Theorem \ref{two weight inequality for fractional maximal operators associated with operators} and Theorem \ref{boundedness auxiliary fractional maximal operator}.

\begin{proof}[Proof of Theorem \ref{two weight inequality for fractional integrals associated with operators}]
By Fatou's lemma, we only need to prove that
\begin{equation*}
|\langle U^{1/q}L^{-{\alpha}/{2}}V^{-1/q}\vec{f},\vec{g}\rangle_{L^2}|\lesssim
[U,V]_{p,q,\alpha,\psi_\theta,\Phi,\Psi}\|\vec{f}\|_{L^p}\|\vec{g}\|_{L^{q'}},
\end{equation*}
where $\vec{f}$, $\vec{g}$ are bounded functions of compact support. Applying Proposition \ref{sparse domination for fractional}, we have
\begin{align*}
&|\langle U^{1/q}L^{-{\alpha}/{2}}V^{-1/q}\vec{f},\vec{g}\rangle_{L^2}|\\
&\quad\lesssim\sum_{t\in\{0,1/3\}^d}\sum_{Q\in\mathcal{D}^t}
\frac{|Q|^{\alpha/d-1}}{\widetilde{\psi}_\theta(Q)}\int_Q\int_Q
|\langle U(x)^{1/q}V(y)^{-1/q}\vec{f}(y),\vec{g}(x)\rangle_{\mathbb{C}^n}|dxdy.
\end{align*}
Now, fix a dyadic grid $\mathcal{D}$, by the generalized H\"{o}lder's inequality and \eqref{rao dong}, we deduce that
\begin{align*}
&\sum_{Q\in\mathcal{D}}
\frac{|Q|^{\alpha/d-1}}{\widetilde{\psi}_\theta(Q)}\int_Q\int_Q
|\langle U(x)^{1/q}V(y)^{-1/q}\vec{f}(y),\vec{g}(x)\rangle_{\mathbb{C}^n}|dxdy\\
&\quad\lesssim\sum_{Q\in\mathcal{D}}\frac{|Q|^{\alpha/d+1}}{\widetilde{\psi}_\theta(Q)}
\Big\|\|U(x)^{1/q}V(y)^{-1/q}\|_{\Phi_y,Q}\Big\|_{\Psi_x,Q}
\|\vec{f}\|_{\bar{\Phi},Q}\|\vec{g}\|_{\bar{\Psi},Q}\\
&\quad=\sum_{Q\in\mathcal{D}}\frac{1}{\widetilde{\psi}_\theta(Q)}|Q|^{\alpha/d+1/q-1/p}
\Big\|\|U(x)^{1/q}V(y)^{-1/q}\|_{\Phi_y,Q}\Big\|_{\Psi_x,Q}|Q|^{1/p-1/q+1}
\|\vec{f}\|_{\bar{\Phi},Q}\|\vec{g}\|_{\bar{\Psi},Q}\\
&\quad\lesssim[U,V]_{p,q,\alpha,\psi_\theta,\Phi,\Psi}\sum_{Q\in\mathcal{D}}|Q|^{1/p-1/q+1}
\|\vec{f}\|_{\bar{\Phi},Q}\|\vec{g}\|_{\bar{\Psi},Q}.
\end{align*}

Let
\begin{equation*}
\mathcal{Q}^k:=\{Q\in\mathcal{D}:a^k<\|\vec{f}\|_{\bar{\Phi},Q}\leq a^{k+1}\},
\end{equation*}
where $a$ is to be determined later. We also let $\mathcal{S}^k$ be the disjoint, maximal collection of cubes $Q\in\mathcal{D}$ such that $a^k<\|\vec{f}\|_{\bar{\Phi},Q}$. Then
\begin{align*}
&\sum_{Q\in\mathcal{D}}|Q|^{1/p-1/q+1}
\|\vec{f}\|_{\bar{\Phi},Q}\|\vec{g}\|_{\bar{\Psi},Q}\\
&\quad\leq\sum_{k\in\mathbb{Z}}a^{k+1}\sum_{P\in\mathcal{Q}^k}|P|^{1/p-1/q+1}\|\vec{g}\|_{\bar{\Psi},P}\\
&\quad\leq\sum_{k\in\mathbb{Z}}a^{k+1}\sum_{Q\in\mathcal{S}^k}\sum_{P\in\mathcal{D}(Q)}
|P|^{1/p-1/q+1}\|\vec{g}\|_{\bar{\Psi},P}.
\end{align*}
By making use of the following equivalent norms between Orlicz average norm and Amemiya norm introduced in \cite{Na}:
\begin{equation*}
\|\vec{f}\|_{\bar{\Phi},Q}\leq\|\vec{f}\|_{\bar{\Phi},Q}'\leq2\|\vec{f}\|_{\bar{\Phi},Q},
\end{equation*}
where
\begin{equation*}
\|\vec{f}\|_{\bar{\Phi},Q}':=\inf\Big\{\lambda+\frac{\lambda}{|Q|}
\int_Q\Phi\Big(\frac{|\vec{f}(y)|}{\lambda}\Big)dy\Big\}.
\end{equation*}
For $Q\in\mathcal{S}^k$, we have
\begin{align*}
&\sum_{P\in\mathcal{D}(Q)}|P|^{1/p-1/q+1}\|\vec{g}\|_{\bar{\Psi},P}\\
&\quad\leq\sum_{P\in\mathcal{D}(Q)}|P|^{1/p-1/q+1}\Big(\lambda+\frac{\lambda}{|P|}
\int_P\bar{\Psi}\Big(\frac{|\vec{g}(x)|}{\lambda}\Big)dx\Big)\\
&\quad=\lambda\sum_{r=0}^\infty\underset{l_P=2^{-r}l_Q}{\sum_{P\in\mathcal{D}(Q)}}
\Big[2^{-(1/p-1/q+1)rd}|Q|^{1/p-1/q+1}\\
&\qquad+2^{-(1/p-1/q)rd}|Q|^{1/p-1/q}
\int_P\bar{\Psi}\Big(\frac{|\vec{g}(x)|}{\lambda}\Big)dx
\Big]\\
&\quad=\lambda\sum_{r=0}^\infty2^{-(1/p-1/q)rd}\Big(|Q|^{1/p-1/q+1}+|Q|^{1/p-1/q}
\int_Q\bar{\Psi}\Big(\frac{|\vec{g}(x)|}{\lambda}\Big)dx\Big)\\
&\quad\sim|Q|^{1/p-1/q+1}\Big(\lambda+\frac{\lambda}{|Q|}
\int_Q\bar{\Psi}\Big(\frac{|\vec{g}(x)|}{\lambda}\Big)dx\Big).
\end{align*}
Choose $\lambda$ such that
\begin{equation*}
\lambda+\frac{\lambda}{|Q|}
\int_Q\bar{\Psi}\Big(\frac{|\vec{g}(x)|}{\lambda}\Big)dx<2\|\vec{g}\|_{\bar{\Psi},Q}'\leq
4\|\vec{g}\|_{\bar{\Psi},Q}.
\end{equation*}
We arrive at
\begin{align*}
&\sum_{Q\in\mathcal{D}}
\frac{|Q|^{\alpha/d-1}}{\widetilde{\psi}_\theta(Q)}\int_Q\int_Q
|\langle U(x)^{1/q}V(y)^{-1/q}\vec{f}(y),\vec{g}(x)\rangle_{\mathbb{C}^n}|dxdy\\
&\quad\lesssim[U,V]_{p,q,\alpha,\psi_\theta,\Phi,\Psi}\sum_{k\in\mathbb{Z}}a^{k+1}\sum_{Q\in\mathcal{S}^k}
|Q|^{1/p-1/q+1}\|\vec{g}\|_{\bar{\Psi},Q}.
\end{align*}

Now we temporarily stop the proof of the conclusion, turn to show that $\mathcal{S}=\bigcup_k\mathcal{S}^k$ is sparse. We adapt the idea in \cite{CRMaPe}. Denote $\Omega_k=\bigcup_{Q\in\mathcal{S}^k}Q$. We are done if
\begin{equation}\label{sparse claim}
|Q\cap\Omega_{k+1}|\leq|Q|/2.
\end{equation}
To prove \eqref{sparse claim}. Note that
\begin{equation*}
|Q\cap\Omega_{k+1}|=\underset{P\subset Q}{\sum_{P\in\Omega_{k+1}}}|P|,~Q\in\mathcal{S}^k,
\end{equation*}
and
\begin{equation*}
a^{k+1}\leq\lambda+\frac{\lambda}{|P|}\int_P\bar{\Phi}\Big(\frac{|\vec{f}(x)|}{\lambda}\Big)dx,~
P\in\Omega_{k+1}.
\end{equation*}
Choose $\lambda_0$ such that
\begin{equation*}
\lambda_0+\frac{\lambda_0}{|Q|}\int_Q\bar{\Phi}\Big(\frac{|\vec{f}(x)|}{\lambda_0}\Big)dx<
4\|\vec{f}\|_{\bar{\Phi},Q}.
\end{equation*}
Hence
\begin{align*}
\underset{P\subset Q}{\sum_{P\in\Omega_{k+1}}}|P|&\leq\underset{P\subset Q}{\sum_{P\in\Omega_{k+1}}}
a^{-k-1}\Big(\lambda_0|P|+\lambda_0\int_P\bar{\Phi}\Big(\frac{|\vec{f}(x)|}{\lambda_0}\Big)dx\Big)\\
&\leq a^{-k-1}\Big(\lambda_0|Q|+\lambda_0\int_Q\bar{\Phi}\Big(\frac{|\vec{f}(x)|}{\lambda_0}\Big)dx\Big)\\
&\leq\frac{2^{d+2}|Q|}{a^{k+1}}\|\vec{f}\|_{\bar{\Phi},\hat{Q}}\leq|Q|/2,
\end{align*}
where $\hat{Q}$ is the dyadic parent of $Q$ and $a=2^{d+3}$. This proves \eqref{sparse claim}.

We return to prove our desired conclusion. Let
\begin{equation*}
\beta/d=1/p-1/s,~ \gamma/d=1/q'-1/s'.
\end{equation*}
Define
\begin{equation*}
M_{A}^\beta f(x):=\sup_{Q\ni x}|Q|^{\beta/d}\|f\|_{A,Q},
\end{equation*}
where $A$ is a Young function. It is known that
\begin{equation*}
M_{A}^\beta :L^p(\mathbb{R}^d)\rightarrow L^s(\mathbb{R}^d),~A\in B_{p,s},
\end{equation*}
see \cite{CrMo}. Since $\mathcal{S}$ is sparse, we have
\begin{align*}
&\sum_{Q\in\mathcal{D}}
\frac{|Q|^{\alpha/d-1}}{\widetilde{\psi}_\theta(Q)}\int_Q\int_Q
|\langle U(x)^{1/q}V(y)^{-1/q}\vec{f}(y),\vec{g}(x)\rangle_{\mathbb{C}^n}|dxdy\\
&\quad\lesssim a[U,V]_{p,q,\alpha,\psi_\theta,\Phi,\Psi}\sum_{Q\in\mathcal{S}}|Q|^{\beta/d+1/s}
\|\vec{f}\|_{\bar{\Phi},Q}|Q|^{\gamma/d+1/s'}\|\vec{g}\|_{\bar{\Psi},Q}\\
&\quad\lesssim[U,V]_{p,q,\alpha,\psi_\theta,\Phi,\Psi}\sum_{Q\in\mathcal{S}}|Q|^{\beta/d}
\|\vec{f}\|_{\bar{\Phi},Q}|Q|^{\gamma/d}\|\vec{g}\|_{\bar{\Psi},Q}|E_Q|\\
&\quad\leq[U,V]_{p,q,\alpha,\psi_\theta,\Phi,\Psi}\int_{\mathbb{R}^d}M_{\bar{\Phi}}^\beta(\vec{f})(x)
M_{\bar{\Psi}}^\gamma(\vec{g})(x)dx\\
&\quad\leq[U,V]_{p,q,\alpha,\psi_\theta,\Phi,\Psi}\|M_{\bar{\Phi}}^\beta\|_{L^s}
\|M_{\bar{\Psi}}^\gamma\|_{L^{s'}}\\
&\quad\lesssim[U,V]_{p,q,\alpha,\psi_\theta,\Phi,\Psi}\|\vec{f}\|_{L^p}\|\vec{g}\|_{L^{q'}}.
\end{align*}
\end{proof}

To prove Theorem \ref{two weight inequality for fractional maximal operators associated with operators}, we first show that the case $1/p-1/q>\alpha/d$ is trivial.
\begin{proposition}
Let $\theta\geq0$, $0<\alpha<d$, $1<p<q<\infty$ and $1/p-1/q>\alpha/d$. Assume that $M_{U,V,\psi_\theta,\alpha}:L^p\rightarrow L^q$ and $V^{1/q}\in L_{loc}^p$. Then $U(x)=0$ for a.e. $x\in\mathbb{R}^d$.
\end{proposition}
\begin{proof}
Fix a cube $Q$ and a vector $\vec{e}$. Set $\vec{f}(y)=V(y)^{1/q}\vec{e}\chi_Q(y)$. From the definition of $M_{U,V,\psi,\alpha}$, for $x\in Q$,
\begin{equation*}
M_{U,V,\psi_\theta,\alpha}\vec{f}(x)\geq\frac{1}{(\psi_\theta(Q)|Q|)^{1-\alpha/d}}\int_Q|U(x)^{1/q}\vec{e}|dy
=\frac{|Q|^{\alpha/d}}{\psi_\theta(Q)^{1-\alpha/d}}|U(x)^{1/q}\vec{e}|.
\end{equation*}
Then the $L^p\rightarrow L^q$ boundedness of $M_{U,V,\psi_\theta,\alpha}$ yields that
\begin{align*}
&\Big(\frac{|Q|^{\alpha q/d}}{\psi_\theta(Q)^{q(1-\alpha/d)}}\int_Q|U(x)^{1/q}\vec{e}|^qdx\Big)^{1/q}\\
&\quad\leq\|M_{U,V,\psi,\alpha}\vec{f}\|_{L^q}\\
&\quad\lesssim\|\vec{f}\|_{L^p}=\Big(\int_Q|V(x)^{1/q}\vec{e}|^pdx\Big)^{1/p}.
\end{align*}
That is
\begin{equation}\label{key for U}
\Big(\frac{1}{|Q|}\int_Q|U(x)^{1/q}\vec{e}|^q\Big)^{1/q}\lesssim|Q|^{1/p-1/q-\alpha/d}
\psi_\theta(Q)^{1-\alpha/d}\Big(\frac{1}{|Q|}\int_Q|V(x)^{1/q}\vec{e}|^p\Big)^{1/p}.
\end{equation}

Let $x_0$ be any Lebesgue point of $|U(x)^{1/q}\vec{e}|^q$ and $|V(x)^{1/q}\vec{e}|^p$. Let $Q_k$ be a sequence of cubes that centered at $x_0$ and shrink to $x_0$. Combining $1/p-1/q-\alpha/d>0$, $1-\alpha/d>0$ and Lebesgue differentiation theorem, the right-hand side of \eqref{key for U} tends to 0. So $|U(x_0)^{1/q}\vec{e}|^q=0$ for any $\vec{e}$. Therefore $U(x_0)=0$.
\end{proof}

\begin{proof}[Proof of Theorem \ref{two weight inequality for fractional maximal operators associated with operators}]
Fix a dyadic grid $\mathcal{D}$. In the proof of Theorem \ref{matrix weight bump condition classical fractional maximal}, it was essentially proved that
\begin{equation*}
\|M_{U,V,\alpha}^{\mathcal{D}}\|_{L^p\rightarrow L^q}\lesssim[U,V]_{p,q,\Phi}^{\mathcal{D}},
\end{equation*}
provided that $[U,V]_{p,q,\Phi}^{\mathcal{D}}<\infty$,
where
\begin{equation*}
M_{U,V,\alpha}^{\mathcal{D}}\vec{f}(x):=\underset{Q\in\mathcal{D}}{\sup_{Q\ni x}}\frac{1}{|Q|^{1-\alpha/d}}\int_Q|U(x)^{1/q}V(y)^{-1/q}\vec{f}(y)|dy
\end{equation*}
and
\begin{equation*}
[U,V]_{p,q,\Phi}^{\mathcal{D}}:=\sup_{Q\in\mathcal{D}}|Q|^{\alpha/d+1/q-1/p}
\Big(-\!\!\!\!\!\!\!\int_Q\|U(x)^{1/q}V^{-1/q}\|_{\Phi,Q}^qdx\Big)^{1/q},~\bar{\Phi}\in B_{p,q}.
\end{equation*}
Denote
\begin{equation*}
[U,V]_{p,q,\alpha,\widetilde{\psi}_{\theta'},\Phi}:=\sup_{Q}\frac{|Q|^{\alpha/d+1/q-1/p}}
{\widetilde{\psi}_{\theta'}(Q)^{1-\alpha/d}}
\Big(-\!\!\!\!\!\!\!\int_Q\|U(x)^{1/q}V^{-1/q}\|_{\Phi,Q}^qdx\Big)^{1/q}.
\end{equation*}
For any $Q\in\mathcal{Q}_r$, one can check that
\begin{equation*}
|Q|^{\alpha/d+1/q-1/p}
\Big(-\!\!\!\!\!\!\!\int_Q\|U(x)^{1/q}V^{-1/p}\|_{\Phi,Q}^qdx\Big)^{1/q}\leq2^{r\theta'(1-\alpha/d)}
[U,V]_{p,q,\alpha,\widetilde{\psi}_{\theta'},\Phi},
\end{equation*}
which implies that
\begin{equation*}
[U,V]_{p,q,\Phi}^{\mathcal{Q}_r}\leq2^{r\theta'(1-\alpha/d)}
[U,V]_{p,q,\alpha,\widetilde{\psi}_{\theta'},\Phi},
\end{equation*}
where
\begin{equation*}
[U,V]_{p,q,\Phi}^{\mathcal{Q}_r}:=\sup_{Q\in\mathcal{Q}_r}|Q|^{\alpha/d+1/q-1/p}
\Big(-\!\!\!\!\!\!\!\int_Q\|U(x)^{1/q}V^{-1/p}\|_{\Phi,Q}^qdx\Big)^{1/q}.
\end{equation*}
Mimic the scheme in the proof of Theorem \ref{matrix weight bump condition classical fractional maximal} and using $\theta>\theta'$, we have
\begin{align*}
\|M_{U,V,\widetilde{\psi}_\theta,\alpha}^{\mathcal{D}}\vec{f}\|_{L^q}&\leq\sum_{r\geq0}
2^{-r\theta(1-\alpha/d)}
\Big\|\sup_{Q\in\mathcal{Q}_r}\frac{1}{|Q|^{1-\alpha/d}}\int_Q|U(x)^{1/q}V(y)^{-1/q}\vec{f}(y)|dy
\Big\|_{L^q}\\
&\lesssim\sum_{r\geq0}2^{-r\theta(1-\alpha/d)}[U,V]_{p,q,\Phi}^{\mathcal{Q}_r}\|\vec{f}\|_{L^p}\\
&\lesssim[U,V]_{p,q,\alpha,\psi_{\theta'},\Phi}\|\vec{f}\|_{L^p}.
\end{align*}
\end{proof}

Finally, we prove Theorem \ref{boundedness auxiliary fractional maximal operator}. We first establish the following necessary lemma.
\begin{lemma}\label{distribution decomp}
Let $1/p-1/q=\alpha/d$ and $(U,V)\in \mathcal{A}_{p,q}^{\rho,\theta/3}$. Then for each $\lambda>0$, there is a disjoint collection of maximal dyadic cubes $\{Q_j\}$ such that
\begin{equation*}
E_\lambda:=\big\{x\in\mathbb{R}^d:M_{\mathcal{U},V,\widetilde{\psi}_\theta,\alpha}^{\mathcal{D}}\vec{f}(x)>
\lambda\big\}=\bigcup_jQ_j
\end{equation*}
and for each $j$,
\begin{equation*}
\lambda<\frac{1}{(\widetilde{\psi}_\theta(Q_j)|Q_j|)^{1-\alpha/d}}
\int_{Q_j}|\mathcal{U}_{Q_j}^{q,q}V(y)^{-1/q}\vec{f}(y)|dy.
\end{equation*}
\end{lemma}
\begin{proof}
Assume that $E_\lambda\neq\emptyset$, otherwise there is nothing to prove. Let $\bar{E}_\lambda$ be the family of dyadic cubes such that
\begin{equation*}
\lambda<\frac{1}{(\widetilde{\psi}_\theta(Q)|Q|)^{1-\alpha/d}}
\int_{Q}|\mathcal{U}_{Q}^{q,q}V(y)^{-1/q}\vec{f}(y)|dy.
\end{equation*}
It is easy to see that the set $\bar{E}_\lambda$ is nonempty. For any $Q\subset\bar{E}_\lambda$, by our assumption, H\"{o}lder's inequality, $\alpha/d-1/p=-1/q$, Lemma \ref{reducing operator} and \eqref{rao dong}, one can check that
\begin{align*}
&\frac{1}{(\widetilde{\psi}_\theta(Q)|Q|)^{1-\alpha/d}}
\int_{Q}|\mathcal{U}_{Q}^{q,q}V(y)^{-1/q}\vec{f}(y)|dy\\
&\quad\leq\frac{|Q|^{\alpha/d}}{\widetilde{\psi}_\theta(Q)^{1-\alpha/d}}
\Big(\frac{1}{|Q|}\int_Q|\mathcal{U}_{Q}^{q,q}V(y)^{-1/q}|_{op}^{p'}\Big)^{1/p'}
\Big(\frac{1}{|Q|}\int_Q|\vec{f}(y)|^pdy\Big)^{1/p}\\
&\quad\lesssim|Q|^{-1/q}\widetilde{\psi}_\theta(Q)^{\alpha/d-1}
|\mathcal{U}_{Q}^{q,q}\widetilde{\mathcal{V}}_{Q}^{q,p'}|_{op}\|\vec{f}\|_{L^p}\\
&\quad\lesssim|Q|^{-1/q}[U,V]_{\mathcal{A}_{p,q}^{\rho,\theta/3}}^{1/q}\|\vec{f}\|_{L^p}.
\end{align*}
It follows that
\begin{equation*}
\lambda<\frac{1}{(\widetilde{\psi}_\theta(Q)|Q|)^{1-\alpha/d}}
\int_{Q}|\mathcal{U}_{Q}^{q,q}V(y)^{-1/q}\vec{f}(y)|dy\rightarrow0
\end{equation*}
as $|Q|\rightarrow\infty$. Now denote the family of such maximal cubes by $\{Q_j\}$ (the subset of $\bar{E}_\lambda$). It is clear that these cubes are pairwise disjoint.

In the end, let us prove $E_\lambda=\bigcup_jQ_j$. Let $x\in E_\lambda$. There is a dyadic cube $Q\ni x$ such that
\begin{equation*}
\lambda<\frac{1}{(\widetilde{\psi}_\theta(Q)|Q|)^{1-\alpha/d}}
\int_{Q}|\mathcal{U}_{Q}^{q,q}V(y)^{-1/q}\vec{f}(y)|dy.
\end{equation*}
So $E_\lambda\subset Q_j$ for some $j$.

On the other hand, let $x\in Q_j$. Since
\begin{equation*}
\lambda<\frac{1}{(\widetilde{\psi}_\theta(Q_j)|Q_j|)^{1-\alpha/d}}
\int_{Q_j}|\mathcal{U}_{Q_j}^{q,q}V(y)^{-1/q}\vec{f}(y)|dy.
\end{equation*}
Then $M_{\alpha,\tilde{\psi}_\theta,\mathcal{U},V}^{\mathcal{D}}\vec{f}(x)>\lambda$. That is, $\bigcup_jQ_j\subset E_\lambda$. Hence, $E_\lambda=\bigcup_jQ_j$.
\end{proof}

\begin{proof}[Proof of Theorem \ref{boundedness auxiliary fractional maximal operator}]
$(1).$ By $\mathcal{A}_Q^{\alpha,\theta}\vec{f}\leq M_{U,V,\psi,\alpha}\vec{f}$, it suffices to prove the result for $M_{U,V,\psi,\alpha}$. Similar to the proof of Proposition \ref{sparse domination for fractional maximal operator}, it suffices to prove $M_{\alpha,\widetilde{\psi},\mathcal{U},V}^{\mathcal{D}}$ is bounded from $L^p$ to $L^{q,\infty}$, where
\begin{equation*}
M_{\alpha,\widetilde{\psi}_\theta,\mathcal{U},V}^{\mathcal{D}}\vec{f}(x):=
\underset{Q\ni x}{\sup_{Q\in\mathcal{D}}}
\frac{1}{(\widetilde{\psi}_\theta(Q)|Q|)^{1-\alpha/d}}\int_Q
|\mathcal{U}_Q^{q,q}V(y)^{-1/q}\vec{f}(y)|dy.
\end{equation*}
By Lemma \ref{reducing operator}, Lemma \ref{distribution decomp} and $1/p-1/q=\alpha/d$,
\begin{align*}
&|\{x\in\mathbb{R}^d:M_{\mathcal{U},V,\widetilde{\psi}_\theta,\alpha}^{\mathcal{D}}\vec{f}(x)>\lambda\}|\\
&\quad\leq\lambda^{-q}\sum_j|Q_j|\Big(\frac{1}{(\widetilde{\psi}_\theta(Q_j)|Q_j|)^{1-\alpha/d}}
\int_{Q_j}|\mathcal{U}_{Q_j}^{q,q}V(y)^{-1/q}\vec{f}(y)|dy\Big)^{q}\\
&\quad\leq\lambda^{-q}\sum_j\widetilde{\psi}_\theta(Q_j)^{q\alpha/d-q}
\Big(-\!\!\!\!\!\!\!\int_{Q_j}|\mathcal{U}_{Q_j}^{q,q}V(y)^{-1/q}|_{op}^{p'}\Big)^{q/p'}
\Big(\int_{Q_j}|\vec{f}(y)|^pdy\Big)^{q/p}\\
&\quad\sim\lambda^{-q}\sum_j\widetilde{\psi}_\theta(Q_j)^{q\alpha/d-q}
|\mathcal{U}_{Q_j}^{q,q}\widetilde{\mathcal{V}}_{Q_j}^{q,p'}|_{op}^q
\Big(\int_{Q_j}|\vec{f}(y)|^pdy\Big)^{q/p}\\
&\quad\lesssim\lambda^{-q}[U,V]_{\mathcal{A}_{p,q}^{\rho,\theta/3}}\|\vec{f}\|_{L^p}^q,
\end{align*}
where we use $p\leq q$ and the cubes $\{Q_j\}_j$ are disjoint in the last inequality.

$(2).$ We only prove the conclusion on the case $p>1$, since the case $p=1$ is almost the same, except that we take the operator norm of the matrix and use Lemma \ref{reducing operator} instead of using duality to obtain the $L^{p'}$ norm. Note that for any $\vec{e}\in\mathbb{C}^n$,
\begin{equation*}
\|\chi_Q\vec{e}\|_{L^{q,\infty}}=|Q|^{1/q}|\vec{e}|.
\end{equation*}
This, together with duality, we have
\begin{align*}
&\sup_Q\sup_{\|\vec{f}\|_{L^p}=1}\Big\|\chi_Q(\psi_\theta(Q)|Q|)^{\alpha/d-1}
\int_Q\mathcal{U}_Q^{q,q}V(y)^{-1/q}\vec{f}(y)dy\Big\|_{L^{q,\infty}}\\
&\quad=\sup_Q\sup_{\|\vec{f}\|_{L^p}=1}\Big||Q|^{\alpha/d+1/q}\psi_\theta(Q)^{\alpha/d-1}
\int_Q\mathcal{U}_Q^{q,q}V(y)^{-1/q}\vec{f}(y)dy\Big|\\
&\quad=\sup_Q\sup_{\|\vec{f}\|_{L^p}=1}\sup_{|\vec{e}|=1}|Q|^{\alpha/d+1/q}\psi_\theta(Q)^{\alpha/d-1}
-\!\!\!\!\!\!\!\int_{Q}\langle\mathcal{U}_Q^{q,q}V(y)^{-1/q}\vec{f}(y),\vec{e}\rangle_{\mathbb{C}^n}dy\\
&\quad=\sup_Q\sup_{|\vec{e}|=1}\sup_{\|\vec{f}\|_{L^p}=1}|Q|^{\alpha/d+1/q-1}\psi_\theta(Q)^{\alpha/d-1}
\int_{\mathbb{R}^d}\langle\vec{f}(y),\chi_QV(y)^{-1/q}\mathcal{U}_Q^{q,q}\vec{e}\rangle_{\mathbb{C}^n}dy\\
&\quad=\sup_Q\sup_{|\vec{e}|=1}|Q|^{\alpha/d+1/q-1}\psi_\theta(Q)^{\alpha/d-1}
\|\chi_QV(y)^{-1/q}\mathcal{U}_Q^{q,q}\vec{e}\|_{L^{p'}}\\
&\quad=\sup_Q\sup_{|\vec{e}|=1}\psi_\theta(Q)^{\alpha/d-1}
\Big(\frac{1}{|Q|}\int_Q|V(y)^{-1/q}\mathcal{U}_{Q}^{q,q}\vec{e}|^{p'}dy\Big)^{1/p'}\\
&\quad\sim\sup_Q\psi_\theta(Q)^{\alpha/d-1}|\widetilde{\mathcal{V}}_{Q}^{q,p'}
\mathcal{U}_Q^{q,q}|_{op}
\sim[U,V]_{\mathcal{A}_{p,q}^{\rho,\theta}}^{1/q}.
\end{align*}
Then $(2)$ follows immediately by the assumption $\mathcal{A}_{Q}^{\alpha,\theta}:L^p\rightarrow L^{q,\infty}$ and $\mathcal{A}_Q^{\alpha,\theta}\vec{f}\leq M_{U,V,\psi_\theta,\alpha}\vec{f}$.
\end{proof}

\section{A characterization of matrix weights}
In this section, we prove Theorem \ref{character weight via average operator}. The idea origins from \cite{CrIsMo}, however, due to the additional critical radius function factor in the definitions of average operator associated with critical radius function, we get narrow ranges of $p,q$.
\begin{proof}[Proof of Theorem \ref{character weight via average operator}]
$(1)\Rightarrow(2)$: \textbf{Case 1. $p>1$:}\quad By H\"{o}lder's inequality, $(U,V)\in \mathcal{A}_{p,q}^{\rho,\theta}$ and $1+q/p'=q-q\alpha/d$, we have
\begin{align*}
\|A_{Q}^{\alpha,\theta}\|_{L^q(U)}^q&=\int_{\mathbb{R}^d}
\Big|\frac{1}{(\psi_\theta(Q)|Q|)^{1-\alpha/d}}
\int_Q\chi_Q(x)U(x)^{1/q}V(y)^{-1/q}V(y)^{1/q}\vec{f}(y)dy\Big|^qdx\\
&\leq\int_Q\frac{|Q|^{\alpha q/d}}{\psi_\theta(Q)^{q-q\alpha/d}}
\Big(-\!\!\!\!\!\!\!\int_Q|U(x)^{1/q}V(y)^{-1/q}|_{op}^{p'}dy\Big)^{q/p'}\\
&\qquad\times\Big(-\!\!\!\!\!\!\!\int_Q|V(y)^{1/q}\vec{f}(y)|^pdy\Big)^{q/p}dx\\
&=\frac{1}{\psi_\theta(Q)|Q|}\int_Q\Big(\frac{1}{\psi_\theta(Q)|Q|}
\int_Q|U(x)^{1/q}V(y)^{-1/q}|_{op}^{p'}dy\Big)^{q/p'}dx\\
&\qquad\times\Big(\int_Q|V(y)^{1/q}\vec{f}(y)|^pdy\Big)^{q/p}\\
&\leq[U,V]_{\mathcal{A}_{p,q}^{\rho,\theta}}\|\vec{f}\|_{L^p(V^{p/q})}^q.
\end{align*}
\textbf{Case 2. $p=1$:}\quad With the aid of Minkowski's inequality and $1-\alpha/d=1/q$,
\begin{align*}
\|A_{Q}^{\alpha,\theta}\|_{L^q(U)}&\leq\int_Q\Big(\frac{1}{\psi_\theta(Q)|Q|}
\int_Q|U(x)^{1/q}V(y)^{-1/q}|_{op}^qdx\Big)^{1/q}|V(y)^{1/q}\vec{f}(y)|dy\\
&\leq[U,V]_{\mathcal{A}_{1,q}^{\alpha,\theta}}^{1/q}\int_Q|V(y)^{1/q}\vec{f}(y)|dy\leq
[U,V]_{\mathcal{A}_{1,q}^{\alpha,\theta}}^{1/q}\|\vec{f}\|_{L^1(V^{1/q})}.
\end{align*}

$(2)\Rightarrow(1)$: \textbf{Case 1. $p>1$:}\quad Lemma \ref{reducing operator} and $1/q+1/p'=1-\alpha/d$ show that
\begin{equation*}
[U,V]_{\mathcal{A}_{p,q}^{\rho,\theta}}^{1/q}\sim\sup_Q\psi_\theta(Q)^{\alpha/d-1}
|\widetilde{\mathcal{V}}_Q^{q,p'}\mathcal{U}_{Q}^{q,q}|_{op}.
\end{equation*}
Hence, we only need to prove that
\begin{equation*}
\sup_Q\psi_\theta(Q)^{\alpha/d-1}
|\widetilde{\mathcal{V}}_Q^{q,p'}\mathcal{U}_{Q}^{q,q}|_{op}<\infty.
\end{equation*}
Now we fix a cube $Q$ and let $\vec{e}\in\mathbb{C}^n$ with $|\vec{e}|=1$. By duality, there exists $g\in L^p(V^{p/q})$ such that
\begin{align*}
|\widetilde{\mathcal{V}}_Q^{q,p'}\mathcal{U}_{Q}^{q,q}\vec{e}|&\sim
\Big(-\!\!\!\!\!\!\!\int_Q|V(y)^{-1/q}\mathcal{U}_{Q}^{q,q}\vec{e}|^{p'}dy\Big)^{1/p'}\\
&=|Q|^{-1/p'}\|\chi_Q\mathcal{U}_{Q}^{q,q}\vec{e}\|_{L^{p'}(V^{-p'/q})}\\
&=|Q|^{-1/p'}\int_Q\langle\mathcal{U}_{Q}^{q,q}\vec{e},\vec{g}(x)\rangle_{\mathbb{C}^n}dx\\
&=|Q|^{1/p}\Big\langle\vec{e},\mathcal{U}_{Q}^{q,q}-\!\!\!\!\!\!\!\int_Q\vec{g}(x)dx\Big
\rangle_{\mathbb{C}^n}\\
&\leq|Q|^{1/p}\Big|\mathcal{U}_{Q}^{q,q}-\!\!\!\!\!\!\!\int_Q\vec{g}(x)dx\Big|\\
&=|Q|^{1/p-\alpha/d}\psi_\theta(Q)^{1-\alpha/d}\Big|\mathcal{U}_{Q}^{q,q}
\frac{1}{(\psi_\theta(Q)|Q|)^{1-\alpha/d}}\int_Q\vec{g}(x)dx\Big|\\
&\sim|Q|^0\psi_\theta(Q)^{1-\alpha/d}\Big(\int_Q|U(y)^{1/q}
A_{Q}^{\alpha,\theta}\vec{g}(y)|^qdy\Big)^{1/q}\lesssim\psi_\theta(Q)^{1-\alpha/d}.
\end{align*}
Therefore, the desired result follows by arranging the above terms.

\textbf{Case 2. $p=1$:}\quad Fix a cube $Q$. Given any $S\subset Q$, where $|S|>0$. Let $\vec{f}=\chi_S(x)V(x)^{-1/q}\vec{e}$ with $\vec{e}\in\mathbb{C}^n$ and $|\vec{e}|=1$. By $(2)$ and $q(1-\alpha/d)=1$, we have
\begin{align*}
|S|\gtrsim\|A_{Q}^{\alpha,\theta}\vec{f}\|_{L^q(U)}&=\Big(\int_Q\Big|U(x)^{1/q}
\Big(\frac{1}{(\psi_\theta(Q)|Q|)^{1-\alpha/d}}\int_SV(y)^{-1/q}\vec{e}dy\Big)\Big|^qdx\Big)^{1/q}\\
&=|S|\psi_\theta(Q)^{\alpha/d-1}\Big(-\!\!\!\!\!\!\!\int_Q\Big|U(x)^{1/q}-\!\!\!\!\!\!\!\int_SV(y)^{-1/q}
\vec{e}dy\Big|^qdx\Big)^{1/q}.
\end{align*}
Thus,
\begin{equation*}
\psi_\theta(Q)^{\alpha/d-1}\Big|\mathcal{U}_{Q}^{q,q}\Big(-\!\!\!\!\!\!\!\int_SV(y)^{-1/q}\vec{e}dy\Big)
\Big|\lesssim1.
\end{equation*}
Applying the Lebesgue differentiation theorem, we arrive at
\begin{equation*}
\esss_{y\in Q}\psi_\theta(Q)^{\alpha/d-1}|\mathcal{U}_{Q}^{q,q}V(y)^{-1/q}|_{op}\lesssim1.
\end{equation*}
Hence, the desired conclusion follows by Lemma \ref{reducing operator}.
\end{proof}

\textbf{Conflict of Interest} The authors declare no conflict of interest.

\textbf{Acknowledgement}
The authors are greatly indebted to the referees for insightful and valuable suggestions which helped to remarkably improve the paper.



\begin{thebibliography}{99}


\bibitem{BiPeWi}K. Bickel, S. Petermichl and B.D. Wick,
Bounds for the Hilbert transform with matrix $A_2$ weights,
J. Funct. Anal. 270(5) (2016), 1719--1743.

\bibitem{Buckley}S.M. Buckley,
Estimates for operator norms on weighted spaces and reverse Jesen inequalities,
Trans. Amer. Math. Soc. 340 (1993), 253--272.

\bibitem{BuiBD}T.A. Bui, T.Q. Bui and X.T. Duong,
Quantitative weighted estimates for some singular integrals related to critical functions,
J. Geom. Anal. 31(10) (2021), 10215--10245.

\bibitem{BuiBD1}T.A. Bui, T.Q. Bui and X.T. Duong,
Quantitative estimates for square functions with new class of weights,
Potential Anal. 57(4) (2022), 545--569.


\bibitem{Cr}D. Cruz-Uribe,
Two weight inequalities for fractional integral operators and commutators,
VI International Course of Mathematical Analysis in Andalusia, (2016), 25--85.

\bibitem{CrIsMo}D. Cruz-Uribe, J. Isralowitz, K. Moen,
Two weight bump conditions for matrix weights. Integral Equations Operator Theory,
90(3) (2018), 31pp.

\bibitem{CRMaPe}D. Cruz-Uribe, J.M. Martell and C. P\'{e}rez,
Weights, extrapolation and the theory of Rubio de Francia. Operator theory: advances and
applications,
Birkh\"{a}user, Basel, 215 (2011), xiv+280pp.

\bibitem{CrMo}D. Cruz-Uribe and K. Moen,
A fractional Muckenhoupt-Wheeden theorem and its consequences,
Integral Equations Operator Theory 76 (2013), 421--446.

\bibitem{CrPe}D. Cruz-Uribe and C. P\'{e}rez,
On the two-weight problem for singular integral operators,
Ann. Sc. Norm. Super. Pisa Cl. Sci. (5) 1(4) (2002), 821--849.

\bibitem{CuPlOu}A. Culiuc, F. Di Plinio and Y. Ou,
Uniform sparse domination of singular integrals via dyadic shifts,
Math. Res. Lett. 25(1) (2018), 21--42.

\bibitem{DiHyLi}F. Di Plinio, T.P. Hyt\"{o}nen and K. Li,
Sparse bounds for maximal rough singular integrals via the Fourier transform,
Ann. Inst. Fourier (Grenoble) 70(5) (2020), 1871--1902.

\bibitem{DuLiYa}X.T. Duong, J. Li and D. Yang,
Variation of Calder\'{o}n-Zygmund operators with matrix weight,
Commun. Contemp. Math. 23(7) (2021), 30pp.

\bibitem{G}M. Goldberg,
Matrix $A_p$ weights via maximal functions,
Pacific J. Math. 211(2) (2003), 201--220.

\bibitem{Hy}T.P. Hyt\"{o}nen,
The sharp weighted bound for general Calder\'{o}n-Zygmund operators.
Ann. of Math. (2) 175(3) (2012), 1473--1506.

\bibitem{HyPeVo}T.P. Hyt\"{o}nen, S. Petermichl and A. Volberg,
The sharp square function estimate with matrix weight,
Discrete Anal. (2019), 8pp.

\bibitem{Is}J. Isralowitz,
Matrix weighted Triebel-Lizorkin bounds: a simple stopping time proof,
Proc. Amer. Math. Soc. 149(10) (2021), 4145--4158.

\bibitem{IM}J. Isralowitz and K. Moen,
Matrix weighted Poincar\'{e} inequalities and applications to degenerate elliptic systems,
Indiana Univ. Math. J. 68(5) (2019), 1327--1377.

\bibitem{IsPoRi}J. Isralowitz, S. Pott and I.P. Rivera-R\'{i}os,
Sharp $A_1$ weighted estimates for vector-valued operators,
J. Geom. Anal. 31(3) (2021), 3085--3116.

\bibitem{LaMPT}M.T. Lacey, K. Moen, C. P\'{e}rez and R.H. Torres,
Sharp weighted bounds for fractional integral operators,
J. Funct. Anal. 259(5) (2010), 1073--1097.

\bibitem{Lerner}A.K. Lerner,
On an estimate of Calder\'{o}n-Zygmund operators by dyadic positive operators,
J. Anal. Math. 121 (2013), 141--161.

\bibitem{Lerner2017}A.K. Lerner,
On pointwise estimates involving sparse operators,
New York J. Math. {\bf 22} (2017), 341--349.

\bibitem{LeLiOmRi}A.K. Lerner, K. Li, S. Ombrosi and I.P. Rivera-R\'{i}os,
On the sharpness of some matrix weighted endpoint estimates,
arXiv: 2310. 06718v1.

\bibitem{LRW}J. Li, R. Rahm and B.D. Wick,
$A_p$ weights and quantitative estimates in the Schr\"{o}dinger setting,
Math. Z. 293(1-2) (2019), 259--283.

\bibitem{M}B. Muckenhoupt,
Weighted norm inequalities for classical operators,
Proc. Symp. Pure Math. 35 (1977), 69--83.

\bibitem{MuRi}P.A. Muller and I.P. Rivera-R\'{i}os,
Quantitative matrix weighted estimates for certain singular integral operators,
J. Math. Anal. Appl. 509(1) (2022), 38pp.

\bibitem{Na}H. Nakano,
Topology of linear topological spaces,
Maruzen Co. Ltd., Tokyo, (1951), viii+281pp.

\bibitem{NaPeTrVo}F. Nazarov, S. Petermichl, S. Treil and A. Volberg,
Convex body domination and weighted estimates with matrix weights,
Adv. Math. 318 (2017), 279--306.

\bibitem{Pe}C. P\'{e}rez,
On sufficient conditions for the boundedness of the Hardy-Littlewood maximal operator between weighted $L^p$-spaces with different weights,
Proc. London Math. Soc. (3) 71(1) (1995), 135--157.

\bibitem{Pe1}C. P\'{e}rez,
Two weighted inequalities for potential and fractional type maximal operators,
Indiana Univ. Math. J. 43 (1994), 663--683.

\bibitem{P}S. Petermichl,
The sharp bound for the Hilbert transform on weighted Lebesgue spaces in terms of the classical $A_p$
characteristic,
Amer. J. Math. 129(5) (2007), 1355--1375.

\bibitem{P1}S. Petermichl,
The sharp weighted bound for the Riesz transforms,
Proc. Amer. Math. Soc. 136(4) (2008), 1237--1249.

\bibitem{S}E.T. Sawyer,
A characterization of a two-weight norm inequality for maximal operators,
Studia Math. 75(1) (1982), 1--11.

\bibitem{T}S. Treil,
Mixed $A_2$-$A_\infty$ estimates of the non-homogeneous vector square function with matrix weights,
Proc. Amer. Math. Soc. 151(8) (2023), 3381--3389.

\bibitem{TV}S. Treil and A. Volberg,
Wavelets and the angle between past and future,
J. Funct. Anal. 143(2) (1997), 269--308.

\bibitem{V}A. Volberg,
Matrix $A_p$ weights via S-functions.
J. Amer. Math. Soc. 10(2) (1997), 445--466.
\end{thebibliography}
\end{document}